\documentclass[11pt,oneside]{amsart}
\usepackage{amssymb, amsmath, amsthm}
\usepackage{ stmaryrd }
\usepackage[normalem]{ulem}
\usepackage{epsfig}
\usepackage{epstopdf}
\usepackage{graphicx}
\usepackage{color}
\usepackage[shortalphabetic]{amsrefs}
\usepackage{mathptmx}
\usepackage{pinlabel}
\usepackage[colorlinks]{hyperref}

\theoremstyle{theorem}
\newtheorem*{theorem*}{Theorem}
\newtheorem*{episode1}{Episode 1}
\newtheorem*{episode2}{Episode 2}
\newtheorem*{episode3}{Episode 3}

\newtheorem*{episode1a}{Episode 1'}
\newtheorem*{episode2a}{Episode 2'}
\newtheorem*{episode3a}{Episode 3'}

      \newtheorem{theorem}{Theorem}[section]
      \newtheorem{challenge}[theorem]{Challenge}

      \newtheorem{example}[theorem]{Example}

      \newtheorem{corollary}[theorem]{Corollary}
      
      \newtheorem{lemma}[theorem]{Lemma}
    
      \newcommand{\R}{{\mathbb R}}
      \newcommand{\C}{\mathbb C}
      \newcommand{\Z}{\mathbb Z}

	\newcommand{\boundary}{\partial}
	
	\newcommand{\mc}[1]{\mathcal{#1}}
	\newcommand{\vect}[1]{\mathbf{#1}}
	
	\newcommand{\defn}[1]{\textbf{#1}}
	\newcommand{\ob}[1]{\overline{#1}}
	\newcommand{\wihat}[1]{\widehat{#1}}
	\newcommand{\im}{\operatorname{im}}

	\newcommand{\diver}{\operatorname{div}}
	
	\newcommand{\grad}{\operatorname{grad}}
	\newcommand{\vl}{\cdot d\vect{s}}
	\newcommand{\scurl}{\operatorname{scurl}}
	\newcommand{\same}{\bumpeq}
	\newcommand{\cint}{\Rbag}
	\newcommand{\ciint}{\Rbag}

	\newcommand{\llpartial}[2]{\frac{\partial #1}{\partial #2}}
	\newcommand{\co}{\mskip0.5mu\colon\thinspace}

\newcommand{\ebb}{\operatorname{ebb}}
\newcommand{\whirl}{\operatorname{whirl}}
\newcommand{\tilt}{\operatorname{tilt}}
\newcommand{\area}{\operatorname{area}}

\title[3 stooges of Vector Calculus]{The 3 stooges of Vector Calculus and their impersonators: A viewer's guide to the classic episodes}
\author{Jennie Buskin}
\author{Philip Prosapio}
\author{Scott A. Taylor}

\begin{document}
\begin{abstract}
The basic theorems of vector calculus are illuminated when we replace the original 3 stooges of vector calculus, Grad, Div, and Curl, with combinatorial substitutes.
\end{abstract}
\maketitle
Grad, Div, and Curl are the three stooges of Vector Calculus: their loveable, but hapless, interactions are viewed with mixtures of delight, puzzlement, and bewilderment by thousands of \sout{students} viewers. The classic episodes involving Grad and Curl include\footnote{Fans of the 3 stooges of Vector Calculus will recognize that Div has been neglected in our list of classic episodes. Throughout this viewer's guide, we invite the reader to fill in the missing episodes which feature this neglected character.}: 
\begin{itemize}
\item[\textbf{Episode 1}] \underline{Horsing Around}: In this hilarious episode, we see that a vector field running in circles does no work if and only if it is a gradient field.
\item[\textbf{Episode 2}] \underline{Violent is the Word for Curl}: Nothing happens when Curl hits Grad who hits a scalar field.
\item[\textbf{Episode 3}] \underline{Grips, Grunts, and Green's}: Mr. Green tries to circulate around a boundary only to find Curl appearing in surprising places.
\end{itemize}

Despite the delight that often greets these classic episodes\footnote{See Section \ref{Scenery and Actors} for precise statements of the episodes/theorems}, audiences often have some difficulty in making sense of the basic plot lines. We maintain that these issues are due, in part, to Grad, Div, and Curl's excellent acting. With great aplomb they manage to combine ideas from calculus, geometry, and topology. In this viewer's guide, we show how the essence of each episode is clarified if we substitute the coarse actors\footnote{``What are the outstanding characteristics of a Coarse Actor?  Firstly I should say a desperate desire to impress.  The true Coarse Actor is most anxious to succeed.  Of course, he is hampered by an inability to act or to move, and a refusal to learn his lines, but no one is more despairing if he fails.  In reality, though, Coarse Actors will never admit that they have done badly.  Law One of Coarse Drama states: `In retrospect all performances are a success''' \cite[page 29]{MG}} Tilt, Ebb, and Whirl for the good actors Grad, Div, and Curl in each of the classic episodes. (Although, as we mentioned, we leave the episodes involving Ebb to the true aficianados.) These coarse actors merely approximate the good actors; they are defined without the use of limits. 

Section \ref{Scenery and Actors} describes the scenery, introduces the actors, and shows how the classic episodes can be reinterpreted so as to make the basic plot lines more evident. In addition to their simplicity, the episodes with the coarse actors have another advantage over the classic episodes: they do not have to consider the possibility that paths intersect infinitely many times. For example, letting \[h(t) = \Big\{\begin{array}{lr} t^2\sin(1/t) \hspace{.2in} & t \neq 0 \\ 0 & t = 0 \end{array}\] for $t \in [0,1]$, the curves $\gamma(t) = (t,0)$ and $\psi(t) = (t, h(t))$  are distinct smooth curves that intersect infinitely often in a neighborhood of the origin. Such curves create challenges that are usually ignored in vector calculus classes\footnote{See for example \cite[page 441]{Colley}}. In the knock-off episodes, however, any two simple curves that intersect infinitely many times actually coincide.

Not only do the knock-off episodes add conceptual clarity, they can also be used to reconstruct the original episodes. Section \ref{Section: Refining} shows that, in fact, the only drawback to the knock-off episodes is their scenery and that by refining the scenery, the acting improves. In more conventional language: by taking limits we can recover the classic episodes from our approximations. In particular, we show that the coarse actor Whirl becomes the good actor Curl and our imitation Episode 3 becomes the  actual Episode 3.

Finally, Section \ref{Cohomology}, shows how the Good Actor's Guild (also known as de Rham cohomology\footnote{Strictly speaking, the cohomology theory given by Grad, Div, and Curl shouldn't be called de Rham cohomology since de Rham cohomology is usually defined using differential forms. However, the two cohomology theories are similar enough that we appropriate the name.}) to which Grad, Div, and Curl belong has a lot in common with the Coarse Actor's Guild (also known as simplicial cohomology) to which Tilt, Ebb, and Whirl belong. These labor unions organize the actors and clarify their working environment.

As always, there's fine print: To simplify matters, we work for the most part in 2-dimensions (although in Section \ref{Cohomology} we point to resources for generalizing these ideas to higher dimensions). Also, there are various topological issues (usually pertaining to the classification of surfaces and the Sch\"onflies theorem) that are ignored. The \textit{cognoscenti} can fill in the missing details without problems, while the new viewer won't notice their absence.

\section{The scenery and actors}\label{Scenery and Actors}
We begin by reviewing the scenery and actors from the classic episodes and then we construct the cheap scenery and introduce the coarse actors.

\subsection{The classic scenery}

The action takes place on a compact, orientable, smooth surface $S$ embedded in $\R^n$. \defn{Smooth} means that there is a tangent plane at every point $x$ of $S$ and the tangent planes vary continuously as $x$ moves around the surface. Figure \ref{Fig: Smooth Surface} shows a smooth surface in $\R^3$. 

\begin{figure}[ht]\label{SmoothSurface}
\includegraphics[scale=.3]{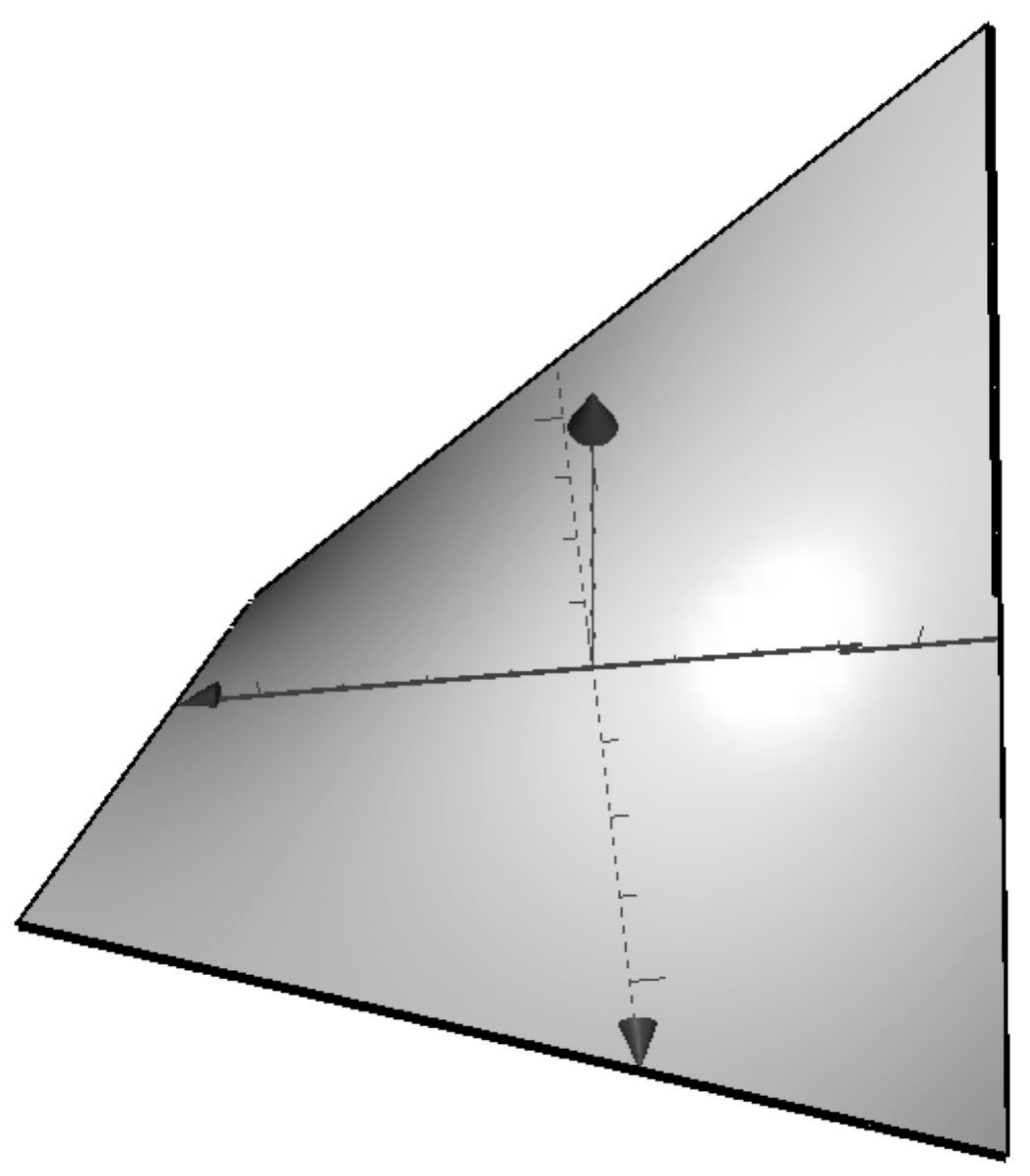}
\caption{A smooth surface in $\R^3$. It is the image of the function $\Phi(s,t) = (s,t,st)$ for $(s,t) \in [-1,1] \times [-1,1]$.}
\label{Fig: Smooth Surface}
\end{figure}

Most of the surfaces that appear in this viewer's guide will be embedded in $\R^2$. At each point of such a surface, the tangent plane coincides with $\R^2$ itself. 

\subsubsection{Fields} A \defn{scalar field} on $S$ is a function $f\co S \to \R$ that is differentiable and whose derivative is continuous. (A differentiable function with continuous derivative is said to be of class C$^1$). A \defn{vector field}\footnote{Vector fields should really take values in the tangent plane to the surface at the points, but we stick with the traditional vector calculus definition.} on $S$ is a C$^1$ function $\vect{F}\co S \to \R^2$. When $n = 2$ (i.e. when $S$ is embedded in the plane), we often write $\vect{F} = \begin{pmatrix} M \\ N \end{pmatrix}$ where $M$ and $N$ are C$^1$ functions from $S$ to $\R$. Generally, a scalar field $f$ on a surface is pictured by shading the surface by making points with large $f$ values light and points with small $f$ values dark. A vector field $\vect{F}$ on a surface is pictured by drawing an arrow based at $x \in S$ pointing in the direction of $\vect{F}(x)$ and of length $||\vect{F}(x)||$. We think of a vector field as telling us the direction of motion and the speed of motion. Figure \ref{SF/VF} shows a scalar field and a vector field on the unit disc.

\begin{figure}[ht]
\includegraphics[scale=.3]{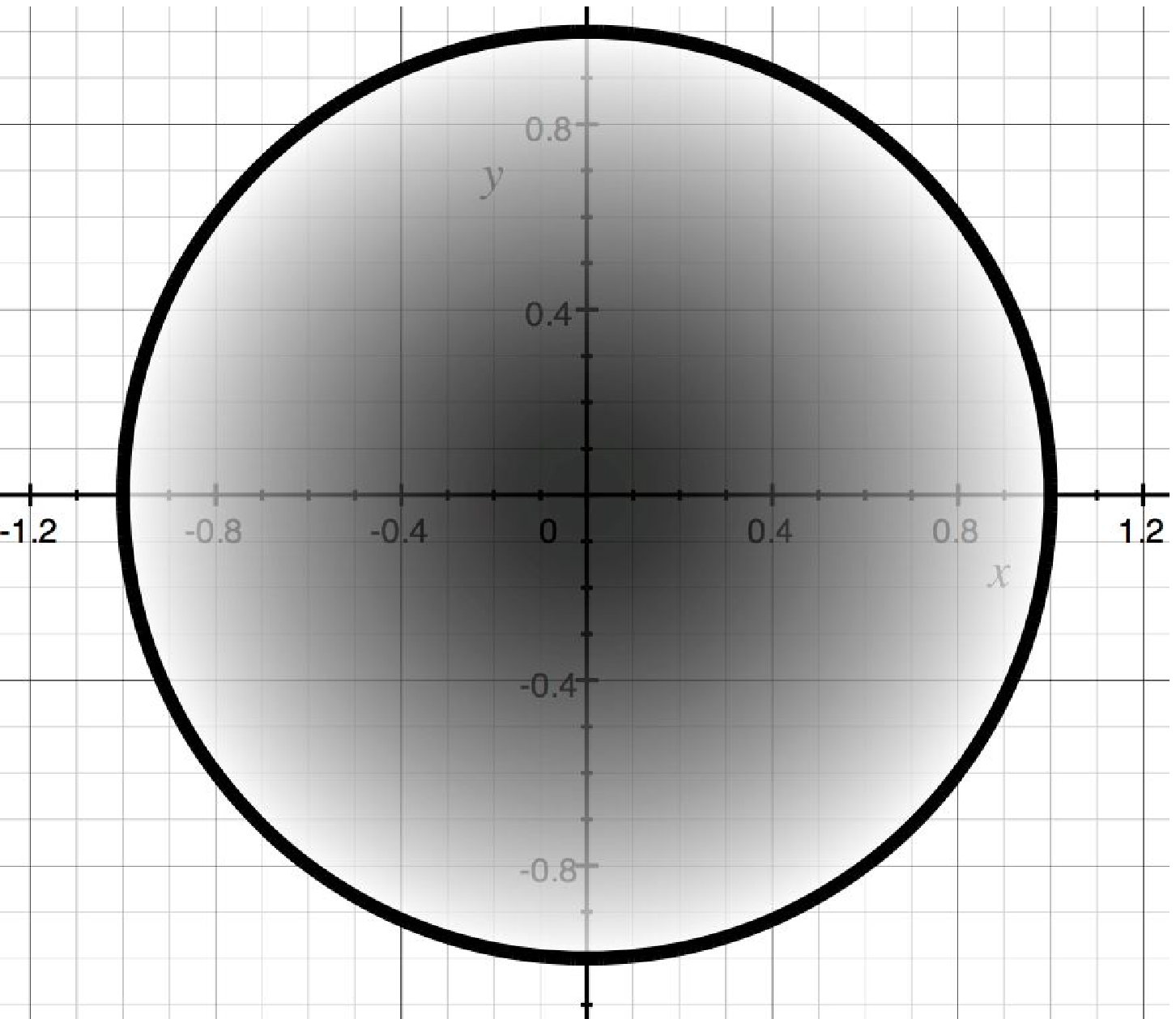}
\includegraphics[scale=.3]{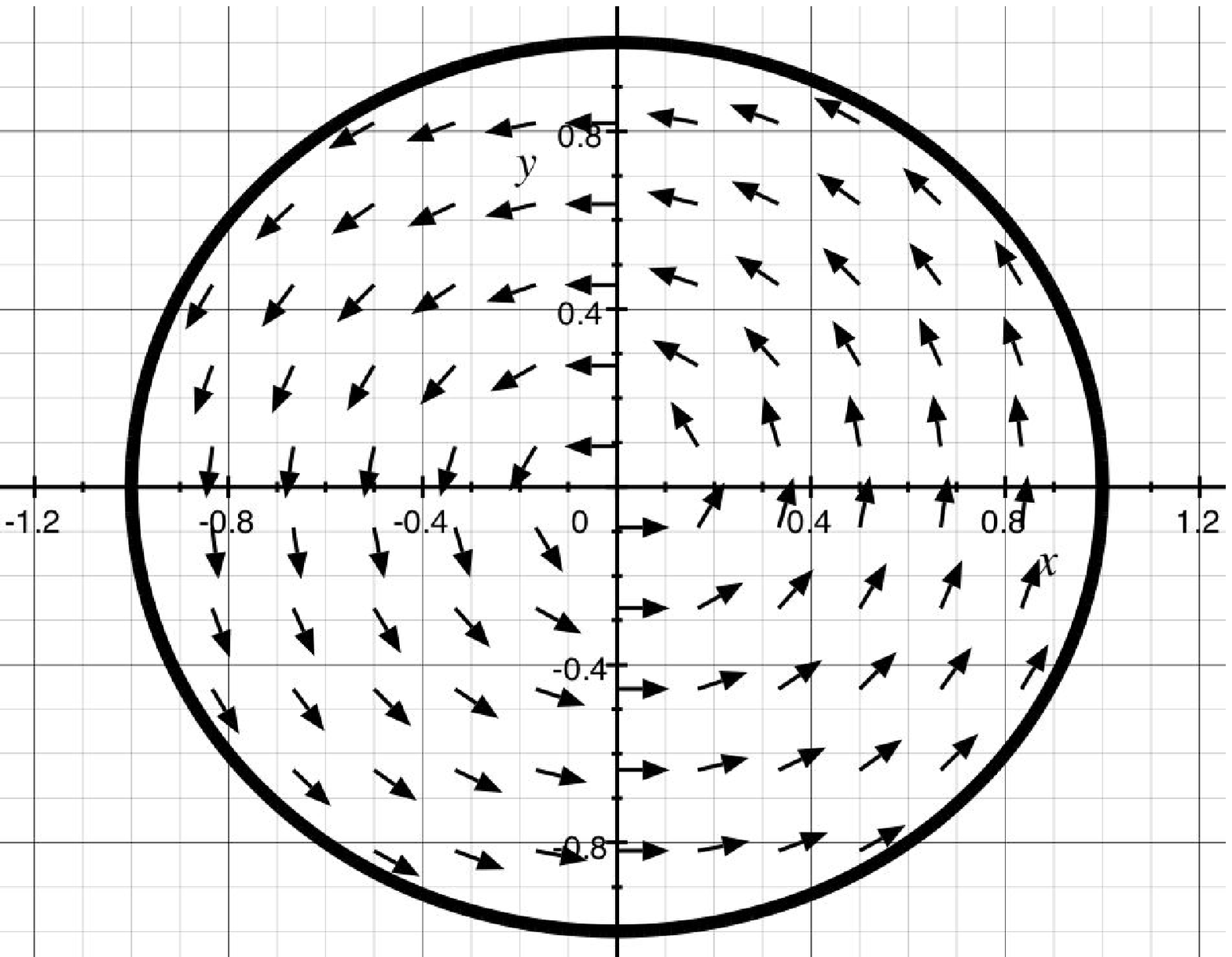}
\caption{On the left is the scalar field $f(x,y) = x^2 + y^2$ on the unit disc in $\R^2$ and on the right is (a portion of) the vector field $\vect{F}(x,y) = (-y,x)/\sqrt{x^2+y^2}$ on the unit disc.}
\label{SF/VF}
\end{figure}

A \defn{path} on $S$ is a continuous function $\phi \co [a,b] \to S$ for some interval $[a,b] \subset \R$. Unless we say otherwise, we require the path to be piecewise C$^1$. That is, there are $t_0, \hdots, t_m \in [a,b]$ with
\[
a = t_0 < t_1 < \hdots < t_{m-1} < t_m = b
\]
such that the restriction of $\phi$ to each subinterval $[t_i, t_{i+1}]$ is continuously differentiable with non-zero derivative at every point. (At the endpoints of the interval we require that the one-sided derivatives exist, are continuous, and are non-zero.) If $\phi$ and $\psi$ are two paths with the same image in $S$, we say that they have the \defn{same orientation} if, as $t$ increases $\phi$ and $\psi$ traverse the image of each C$^1$ portion of the image in the same direction. Otherwise, we say that $\phi$ and $\psi$ have \defn{opposite orientations}. We indicate the orientation by drawing arrows on the image of the path as in Figure \ref{PathFig}. The path $\phi$ is a \defn{closed curve} if $\phi(a) = \phi(b)$ and the path $\phi$ is \defn{embedded} if it is one-to-one on the intervals $[a,b)$ and $(b,a]$.

\begin{figure}[ht]
\includegraphics[scale=.3]{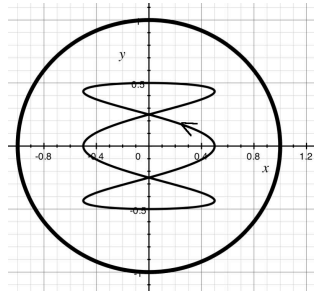}
\caption{The image of the path $\phi(t) = (.5\cos(3t), .5\sin(t))$ for $0 \leq t\leq 2\pi$ in the unit disc. It is a non-embedded closed curve. The orientation of the path is marked with an arrow.}
\label{PathFig}
\end{figure}

\subsubsection{Integrals}

The \defn{integral} of a scalar field $f$ over a path $\phi\co[a,b] \to S$ is defined to be
\[
\int_\phi f \thinspace ds = \int_a^b f(\phi(t)) ||\phi'(t)|| \thinspace dt.
\]
We think of it as measuring the ``total amount'' of $f$ on the path $\phi$. 

The \defn{integral} of a vector field $\vect{F}$ over $\phi$ is defined to be
\begin{equation}\label{Defn vl}
\int_\phi \vect{F} \vl = \int_a^b \vect{F}(\phi(t))\cdot \phi'(t) \thinspace dt.
\end{equation}
We think of it as measuring the total amount that $\vect{F}$ points in the same direction as $\phi$. The integral also measures the amount of ``work'' done by the vector field $\vect{F}$ on an object travelling along the path $\phi$. The classic version of \hyperref[Episode 1]{Episode 1} shows that vector fields produced by Grad don't do any work on closed curves.  A straightforward (and standard) calculation shows that if we replace $\phi$ by another path $\psi$ with the same image, then the integrals of $\vect{F}$ over $\phi$ and $\psi$ will be the same if $\phi$ and $\psi$ have the same orientation and they will be the same in absolute value, but opposite in sign, if $\phi$ and $\psi$ have opposite orientations.

\subsection{The good actors and their classic episodes}

\textit{Lights! Camera! Action!} As the scene opens, only a surface\footnote{Of course, much of what follows can be done in higher dimensions, but for simplicity we stick to surfaces in the plane.} $S \subset \R^2$, a scalar field $f$ on $S$ and a vector field $\vect{F}= \begin{pmatrix} M \\ N \end{pmatrix}$ on $S$  are visible. After a brief moment, our heros enter -- in disguise!

Grad is the first, and best known, of the three stooges. It is the vector field on $S$ defined by the equation
\begin{equation}\label{Grad Def}
\grad f = \begin{pmatrix} \llpartial{f}{x} \\ \llpartial{f}{y} \end{pmatrix}.
\end{equation}

Div is the scalar field on $S$ defined by
\begin{equation}\label{Div Disguised}
\diver \begin{pmatrix} M \\ N \end{pmatrix} = \llpartial{M}{x} + \llpartial{N}{y}.
\end{equation}

Curl\footnote{Usually, Curl is defined so that it is a vector-valued function. Since the classic episodes only involve surfaces, however, we simply use scalar curl.}  is defined by:
\begin{equation}\label{Disguised Curl}
\scurl \begin{pmatrix} M \\ N \end{pmatrix} = \llpartial{N}{x} - \llpartial{M}{y}.
\end{equation}

Although useful for computations, these definitions disguise the true identities of Grad, Div, and Curl since they give no indication of what is being measured, or why such expressions are useful. 

Grad's disguise is easiest to remove. Standard manipulations (see, for example, \cite[page 161]{Colley}) with the definition of the directional derivative, show that the instantaneous rate of change of the scalar field $f$ in the direction of a unit vector $\vect{u}$ is equal to $\grad f \cdot \vect{u}$. This quantity is maximized when $\vect{u}$ points in the same direction as $\grad f$.

Div's and Curl's disguises are much harder to remove. Indeed, not until after \hyperref[Episode 3]{Episode 3} are their true identities revealed\footnote{Particularly attentive viewers may see references to their identities earlier, but those early references, although helpful to plot development, distort the logical sequence of events.}. But we, the authors of this viewer's guide, do not hesitate to spoil the surprise. 

To describe their true identities: let $\vect{a} \in S$ and let $\gamma_n$ be a sequence of piecewise C$^1$ curves enclosing the point $\vect{a}$, oriented counter-clockwise and converging to $\vect{a}$ as $n \to \infty$. Let $A_n$ be the area enclosed by $\gamma_n$ and let $\vect{N}$ be the outward pointing unit normal vector to $\gamma_n$, as in Figure \ref{Converging Curves}. (The vector $\vect{N}$ depends on the points of $\gamma_n$ and is one of the two unit vectors orthogonal to the tangent vector to $\gamma_n$.)

\begin{figure}[ht]
\labellist \small\hair 2pt 
\pinlabel {$\gamma_1$} [tr] at 113 99
\pinlabel {$\gamma_2$} [tl] at 206 110
\pinlabel {$\gamma_3$} [tl] at 174 148
\pinlabel {$\stackrel{\vect{N}}{\rightarrow}$} [l] at 284 131
\endlabellist 
\centering 
\includegraphics[scale=0.45]{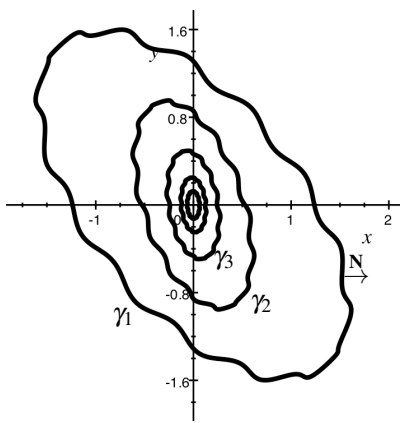}
\caption{A sequence of curves $\gamma_n$ converging to the origin. The vector $\vect{N}$ for one point on $\gamma_1$ is shown.}
\label{Converging Curves}
\end{figure}

The true identity of Div at $\vect{a}$ is then:
\begin{equation}\label{Div Revealed}
\diver \vect{F}(\vect{a}) = \lim_{n \to \infty} \frac{1}{A_n}\int_{\gamma_n} \vect{F}\cdot \vect{N} \thinspace ds,
\end{equation}
and the true identity of Curl is:
\begin{equation}\label{Curl Revealed}
\scurl \vect{F}(\vect{a}) = \lim_{n \to \infty} \frac{1}{A_n}\int_{\gamma_n} \vect{F} \vl.
\end{equation}
Succinctly, we say that $\diver \vect{F} (\vect{a})$ is the infinitessimal rate of expansion of $\vect{F}$ at $\vect{a}$ and that $\scurl \vect{F}(\vect{a})$ is the infinitessimal rate of circulation of $f$ at $\vect{a}$. Some well-meaning authors (e.g. \cite{Schey}) attempt to re-order the episodes of the original season and \emph{define} Div and Curl using Equations \eqref{Div Revealed} and \eqref{Curl Revealed} in place of Equations \eqref{Div Disguised} and \eqref{Disguised Curl}. There are, however, significant issues with this approach (namely: why does the limit exist and why is it independent of the sequence $(\gamma_n)$?) Our coarse actors will have no such issues.

Here are the classic episodes involving Grad and Curl. Episode 1 concerns Grad; Episode 3 concerns Curl; and Episode 2 concerns their relationship.

\begin{episode1}[Horsing Around]
Suppose that $\vect{F}$ is a C$^1$ vector field on a surface $S \subset \R^n$. Then there is a scalar field $f\co S \to \R$ such that $\vect{F} = \grad f$ if and only if $\int_{\phi} \vect{F} \vl = 0$ for every closed C$^1$ curve $\phi$ in $S$.
\end{episode1}

\begin{episode2}[Violent is the Word for Curl]
Suppose that $f\co S \to \R$ is a scalar field on a surface $S\subset\R^2$ such that all second partial derivatives exist and are continuous\footnote{i.e. $f$ is of class C$^2$}, then $\scurl(\grad f) = 0$.
\end{episode2}

\begin{episode3}[Grips, Grunts, and Green's\footnote{Traditionally known as Green's Theorem}]\label{Episode 3}
For any compact surface $S \subset \R^2$ and any C$^1$ vector field $\vect{F}$ on $S$,
\[
\int_{\boundary S} \vect{F} \vl = \iint_{S} \scurl \vect{F} \thinspace dA
\]
where $\boundary S$ is given the orientation so that $S$ is always on the left.
\end{episode3}

The \sout{proofs} plots of these classic episodes are complicated by Curl's disguise, the fact that C$^1$ curves can intersect infinitely many times,  and the reliance on unenlightening calculations.  Typically, (e.g. in \cites{Colley, MarsdenTromba}) Episode 3, for example, is proven only for regions that can be decomposed in a relatively nice way into finitely many so-called Type III regions. The proof for vector fields on Type III regions consists of two tedious calculations relying on the definition of line integral and Fubini's theorem. Colley \cite{Colley} references sources for more general proofs, but those proofs are rather difficult to follow in the context of Green's theorem. Apostol, in the first, but not second, edition of his book \cite[Theorem 10.43]{Apostol} provides the director's cut of Episode 3. His version relies on a rather difficult decomposition of the surface into finitely many Type I and Type II regions.  Our approach (via \hyperref[Whirl Theorem]{Episode 3'}) avoids Fubini's theorem and makes direct use of Riemann sums. It applies most directly to surfaces having boundary that is the union of finitely many vertical and horizontal line segments (so-called ``V/H surfaces''), but we also show how the Change of Variables theorem can be used to obtain a version of Episode 3 for a much wider class of surfaces.

\subsection{Cheap scenery}
Unlike the excellent scenery of the original episodes, the scenery provided for our coarse actors exhibits its structural elements. 

A \defn{combinatorial surface} $(S,G)$ (see Figure \ref{ExampleCombSurf} for an example) consists of a compact, smooth surface $S \subset \R^n$ and a graph embedded in $S$. The graph consists of finitely many vertices $\mc{V}$ (i.e. points in $S$); finitely many edges $\mc{E}$, each of which is an embedded C$^1$ curve with endpoints on vertices; and faces $\mc{F}$, each of which is the closure of a component of $S - (\bigcup \mc{V} \cup \bigcup \mc{E})$.  For simplicity, we require the following:
\begin{itemize}
\item each face in $\mc{F}$ is homeomorphic to a disc
\item each vertex is the endpoint of at least one edge in $\mc{E}$.
\item $\boundary S \subset \bigcup \mc{E}$.
\item Each edge not contained in $\boundary S$ lies in the boundary of two distinct faces.
\end{itemize}

The first condition ensures that each face is topologically simple. The second will ensure that our coarse actor stand-in for Grad is well-defined. The third allows us to relate the behaviour of combinatorial versions of scalar fields and vector fields on the interior of $S$ to their behaviour on the boundary. The final condition isn't strictly necessary, but it simplifies the exposition. Figure \ref{ExampleCombSurf} shows the surface $S$ from Figure \ref{Fig: Smooth Surface} with an embedded graph $G \subset S$ such that $(S,G)$ is a combinatorial surface. Figures \ref{Fig: Oriented Edge} and \ref{Fig: EdgePath} show examples of combinatorial surfaces in $\R^2$. (Ignore, for the moment, the arrows and labels in those figures.)

\begin{figure}[ht]\label{ExampleCombSurf}
\includegraphics[scale=.3]{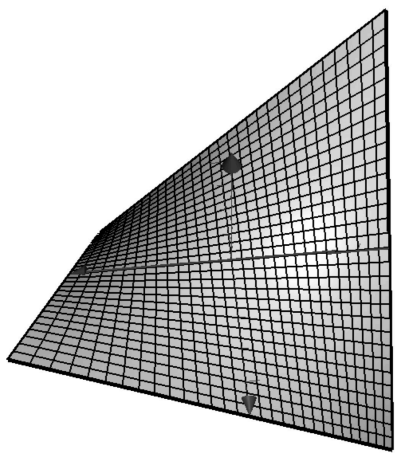}
\caption{An example of a combinatorial surface $(S,G)$. The graph $G$ includes the boundary of the surface. The vertices are the points where the curves intersect. The edges are the portions of the curves between the vertices, and the faces are the closures of the 2-dimensional pieces in the complement of the edges and vertices.}
\label{ExampleCombSurf}
\end{figure}

\subsubsection{Orientations} The notion of orientation plays an important role in topology, geometry, and vector calculus. In our combinatorial setting, we need to consider several types of orientation. Each is based on the notion of ``orienting'' an edge.

An \defn{orientation} of an edge in $\mc{E}$ is a choice of arrow pointing along the edge. If the orientation points into an endpoint $v \in \mc{V}$, then $v$ is called the \defn{sink} of the oriented edge and if the orientation points out of $v$, then $v$ is called the \defn{source} of the edge. Figure \ref{Fig: Oriented Edge} shows an oriented edge. If an edge forms a loop, then its endpoints coincide and $v$ is both the source and the sink of the edge.  

\begin{figure}[ht]
\labellist \small\hair 2pt 
\pinlabel {$v$} [br] at 139 374
\pinlabel {$w$} [l] at 296 197
\pinlabel {$e$} [bl] at 220 283
\endlabellist 
\centering 
\includegraphics[scale=.3]{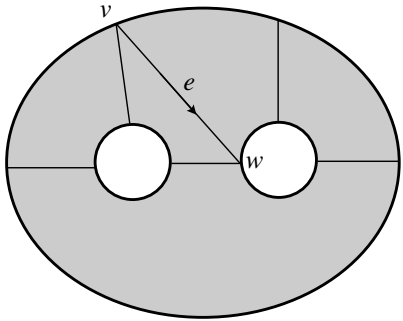}
\caption{The edge $e$ has been oriented so that $v$ is its source and $w$ is its sink.}
\label{Fig: Oriented Edge}
\end{figure}

An orientation of a face $\sigma$ is a choice of orientation on each edge of $\boundary \sigma$ so that the sink of each edge is the source for another edge in $\boundary \sigma$. The orientations on the edges of $\boundary \sigma$ are said to be \defn{induced} by the orientation of $\sigma$. The combinatorial surface $(S,G)$ is \defn{oriented} if each face of $G$ is given an orientation such that if faces $\sigma$ and $\tau$ coincide along an edge $e$ of $G$ then the orientations of $\sigma$ and $\tau$ induce opposite orientations on $e$. Figure \ref{Fig: OrientedSurface} shows an example of an oriented surface. If $(S,G)$ is oriented, then the orientations of the faces adjacent to the edges on $\boundary S$ induce an orientation on those edges. This is called the orientation of $\boundary S$ \defn{induced} by the orientation of $(S,G)$. Not every combinatorial surface can be given an orientation: see,  for example, the M\"obius band in Figure \ref{Fig: Moebius}. When our surface is in $\R^2$, we will always assume that it has been given the \defn{standard orientation}, where the boundary of every face is oriented so that the face is on the left when viewed from above.

\begin{figure}[ht]
\includegraphics[scale=.3]{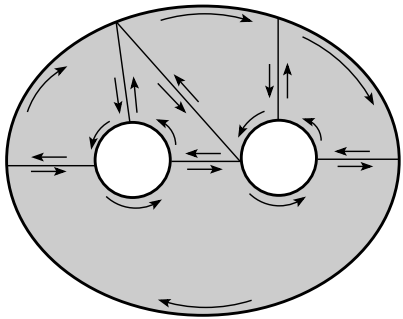}
\caption{Each face of the graph has been given an orientation. Since any two faces adjacent along an edge give opposite orientations to the edge, this is an orientation of the surface. Note that each component of the boundary of the surface inherits an induced orientation from the orientation on the faces.}
\label{Fig: OrientedSurface}
\end{figure}

\begin{figure}[ht]
\includegraphics[scale=.3]{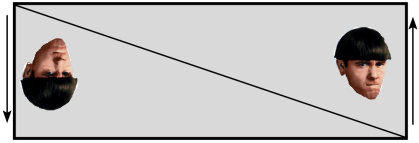}
\caption{The M\"obius (or is it Moe-bius?) band is the surface obtained by gluing the left edge of the rectangle to the right edge of the rectangle so that the arrows match. It is impossible to orient the two faces of the combinatorial surface so as to give an orientation of the M\"obius band.}
\label{Fig: Moebius}
\end{figure}

\subsubsection{Paths and Loops} A finite sequence of oriented edges $e_1, \hdots, e_n$ is called a \defn{edge path} if, for each $1 \leq i \leq n$, the sink vertex of $e_i$ is the source vertex of $e_{i+1}$. It is an \defn{edge loop} if the sink vertex for $e_n$ is the source vertex for $e_1$. An edge path is \defn{embedded} if, for all $i$, $e_i$ and $e_{i+1}$ are distinct edges in $G$ and if no vertex is the sink vertex for more than one edge in the sequence. Figure \ref{Fig: EdgePath} shows an embedded edge loop and a non-embedded edge path. 

\begin{figure}[ht]
\labellist \small\hair 2pt 
\pinlabel {$e_1$} [bl] at 137 279
\pinlabel {$e_2$} [b] at 206 176
\pinlabel {$e_3$} [l] at 239 205
\pinlabel {$e_4$} [br] at 206 244
\pinlabel {$e_5$} [br] at 142 160
\pinlabel {$e_6$} [b] at 93 155
\pinlabel {$v$} [br] at 115 311
\pinlabel {$w$} [br] at 65 144
\endlabellist 
\centering 
\includegraphics[scale=.4]{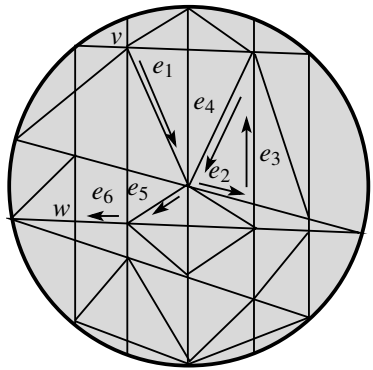}
\caption{The edges $e_1, e_2, \hdots, e_6$ are a non-embedded edge path from vertex $v$ to vertex $w$. The edges $e_2, e_3, e_4$ are an embedded edge loop.}
\label{Fig: EdgePath}
\end{figure}

\subsubsection{Fields}
The scenery in our bad productions is so flimsy that we actually need two different kinds of scalar fields. We call them, rather unimaginatively, ``vertex scalar fields'' and ``face scalar fields''. A \defn{vertex scalar field (VSF)} is a function $f \co \mc{V} \to \R$, while a \defn{face scalar field (FSF)} is a function $f\co \mc{F} \to \R$. We think of scalar fields as telling us the amount of something (perhaps cream pies?) on a vertex or face (Figures \ref{Fig: PieFight1} and \ref{Fig: VSFFSF}.)

\begin{figure}[tbh]
\centering
\includegraphics[scale=0.2]{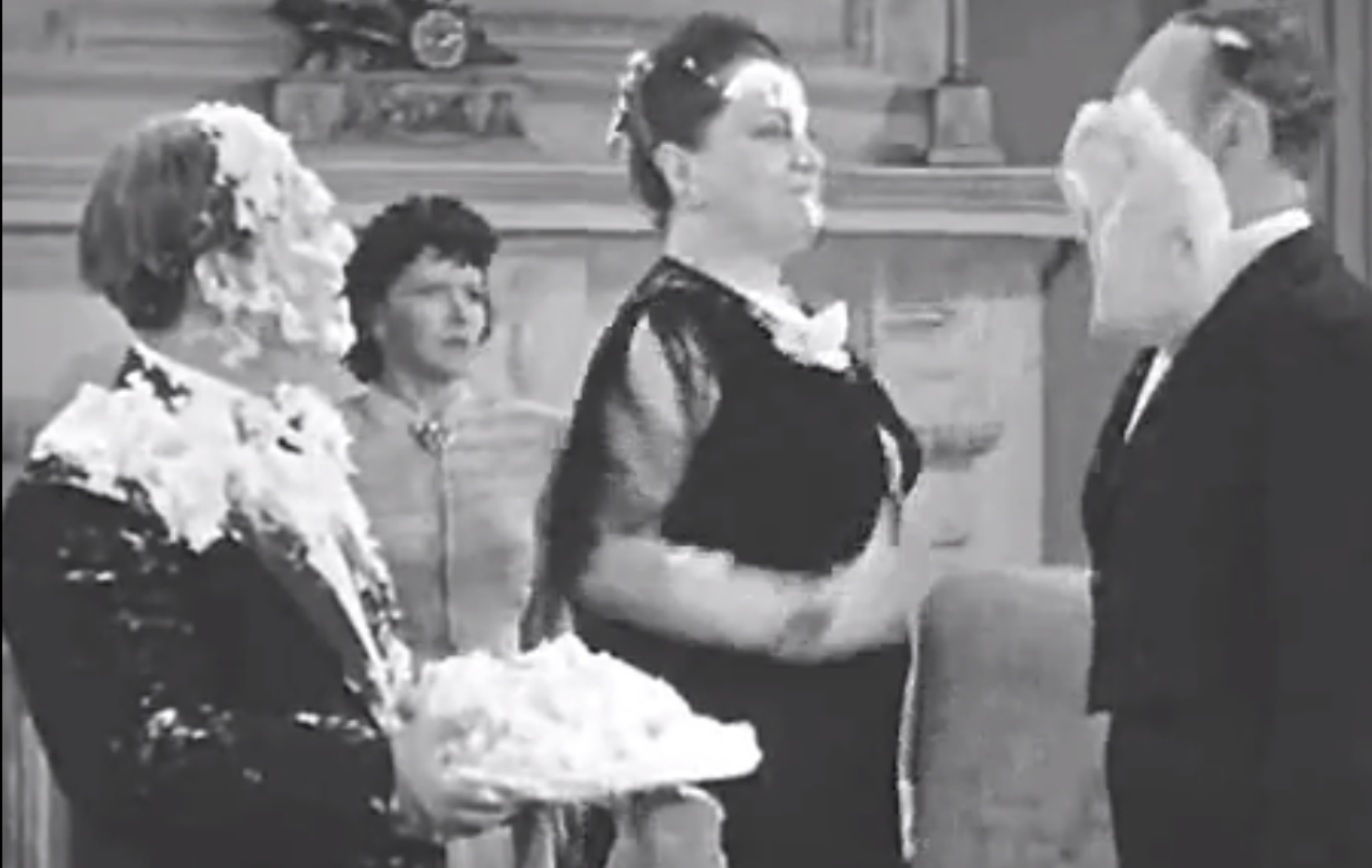}
\caption{A face scalar field will tell us the number of cream pies on each face.} \label{Fig: PieFight1}
\end{figure}

\begin{figure}[ht]
\labellist \small\hair 2pt 
\pinlabel {$-3$} [r] at 0 191
\pinlabel {$.5$} [br] at 140 371
\pinlabel {$7$} [bl] at 344 377
\pinlabel {$2$} [l] at 496 198
\pinlabel {$0$} [r] at 388 198
\pinlabel {$19$} [l] at 299 196
\pinlabel {$2$} [r] at 205 196
\pinlabel {$-1$} [l] at 115 191
\pinlabel {$5$} [t] at 158 242
\pinlabel {$4$} [t] at 344 245
\pinlabel {$\underline{\sqrt[3]{13}}$} at 251 91
\pinlabel {$\underline{e^\pi}$} at 82 260
\pinlabel {$\underline{10}$} at 198 258
\pinlabel {$\underline{777}$} at 410 266
\pinlabel {$\underline{-1}$} at 265 307
\endlabellist 
\centering 
\includegraphics[scale=.3]{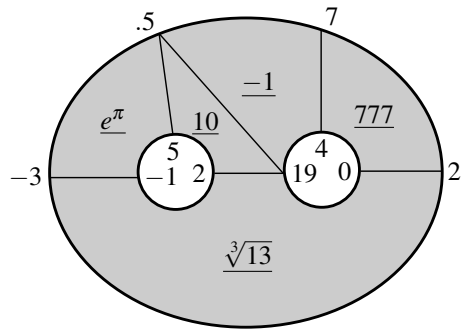}
\caption{The underlined numbers represent a FSF and the non-underlined numbers represent a VSF.}
\label{Fig: VSFFSF}
\end{figure}

Like vector fields on smooth surfaces, a \defn{combinatorial vector field} (CVF) gives both a direction and a rate of motion. In the case of a CVF, however, the motion is confined to the edges of $G$. More formally, a CVF on a combinatorial surface $(S,G)$ is defined to be a choice of orientation on each edge together with a function $\vect{F}\co \mc{E} \to \R$. Planting a cream pie firmly in notation's visage, we refer to \emph{both} the choice of orientation and the function as $\vect{F}$. We think of a CVF as telling us a direction and rate of movement. For example if Moe and Larry are standing on the endpoints of an edge, a CVF tells us whether the cream pie is flying from Moe to Larry or from Larry to Moe and what its rate of travel is.

For a moving object, negating the rate of motion has the same effect as reversing the direction of motion. Similarly, we say that two CVFs $\vect{F}$ and $\vect{G}$ are the \defn{same}, and we write $\vect{F} \same \vect{G}$, if whenever $\vect{F}$ and $\vect{G}$ assign the same orientation to an edge $e$, then $\vect{F}(e) = \vect{G}(e)$ and whenever $\vect{F}$ and $\vect{G}$ assign opposite orientations to $e$, then $\vect{F}(e) = - \vect{G}(e)$. The relation $\same$ is an equivalence relation on CVFs. 

\begin{figure}[ht]
\labellist \small\hair 2pt 
\pinlabel {$3.5$} [br] at 48 315
\pinlabel {$6.5$} [b] at 247 393
\pinlabel {$5$} [l] at 445 318
\pinlabel {$5$} [t] at 247 2
\pinlabel {$6$} [br] at 121 230
\pinlabel {$3$} [bl] at 200 230
\pinlabel {$3$} [t] at 160 147
\pinlabel {$15$} [br] at 305 230
\pinlabel {$4$} [bl] at 386 230
\pinlabel {$19$} [t] at 343 147
\pinlabel {$2$} [t] at 55 189
\pinlabel {$4.5$} [l] at 159 296
\pinlabel {$17$} [t] at 248 189
\pinlabel {$18.5$} [bl] at 211 296
\pinlabel {$3$} [r] at 345 296
\pinlabel {$2$} [t] at 440 189
\endlabellist 
\centering 
\includegraphics[scale=.4]{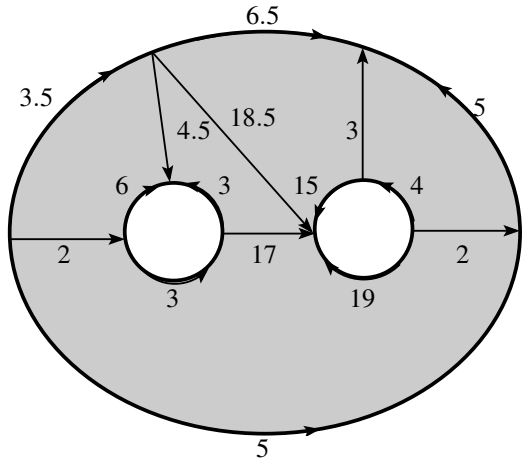}
\caption{A combinatorial vector field.}
\label{Fig: CVF}
\end{figure}

If we fix an orientation on each edge of $G$, a combinatorial vector field is \defn{canonical} if the orientation of each edge given by the CVF is the same as the given orientation. For a fixed orientation on each edge of $G$, there is, in each equivalence class of CVFs, exactly one canonical CVF. 

\subsubsection{Combinatorial Integrals}
In their first calculus course, students learn that integrals are approximated by certain sums. Likewise, the well-known integrals of vector calculus are approximated by certain sums. We denote these sums with the symbol $\cint$ to emphasize the analogy with integrals.

The \defn{integral} of a VSF $f$ over a multiset\footnote{A multiset is similar to a set, except that elements may appear more than once. We require all multisets to be finite.} $V$ of vertices of $G$ is defined to be 
\[
\cint_V f = \sum_{v \in V} f(v).
\]
If each vertex represents a table with cream pies stacked on it, the integral of a VSF over a set of tables is just the total number of cream pies stacked on all those tables.

Similarly, the \defn{integral} of an FSF $f$ over a multiset $\Phi$ of faces of $G$ is defined to be:
\[
\ciint_\Phi f = \sum_{F \in \Phi} f(F).
\]
The integral of a FSF tells us the total number of cream pies on the faces in $\Phi$.

If $e$ is an oriented edge and if $\vect{F}$ is a CVF, define $\epsilon(e,\vect{F})$ to be +1 if the given orientation of $e$ and the orientation of $e$ by $\vect{F}$ are the same and define it to be $-1$ if they differ. If $E$ is a multiset of edges, define the \defn{integral} of $\vect{F}$ over $E$ by
\[
\cint_E \vect{F} = \sum_{e \in E} \epsilon(e,\vect{F})\vect{F}(e).
\]
Notice that if $\vect{F} \same \vect{G}$, then $\cint_E \vect{F} = \cint_E \vect{G}$ since if $\vect{F}$ and $\vect{G}$ differ on an edge $e$, then they assign $e$ opposite orientations and take the same value with opposite signs on $e$. Ignoring the distinction between sequences and multisets, we can consider the integral of a CVF over a path. It measures the total velocity of cream pies flying over the path. The integrals of $\vect{F}$ over a path $e_1, \hdots, e_n$ and over its reversal $e_n, \hdots, e_1$ are equal in absolute value, but opposite in sign.

\subsection{The coarse actors}

With appallingly bad taste, we now replace the original 3 stooges Grad, Div, and Curl with 3 poor substitutes: Tilt, Ebb, and Whirl. They only approximate the originals, but as we shall see in Section \ref{Section: Refining}, with some refinement they become better.

\subsubsection{Tilt} Tilt is Grad's understudy. Unlike the original, Tilt cannot act on all scalar fields, only on vertex scalar fields. If $f$ is a VSF on the combinatorial surface $(S,G)$, we define the CVF $\tilt f$ as follows. If $e$ is an edge of $G$ with endpoints $v$ and $w$ such that $f(v) \leq f(w)$, let $\tilt f$ give $e$ the orientation that points from $v$ to $w$ and define $\tilt f(e) = f(w) - f(v)$. Notice that $\tilt f (e)$ is well-defined unless $f(v) = f(w)$, in which case there are two possible orientations of $e$, but with either orientation we still have $\tilt f(e) = 0$. If $f$ tells us the number of cream pies at each vertex, then $\tilt f$ is a CVF that sends cream pies along edges from vertices with fewer cream pies to vertices with more cream pies such that the larger the difference between the number of cream pies at the endpoints of an edge, the faster the cream pies move. The CVF in Figure \ref{Fig: CVF} is the tilt of the VSF in Figure \ref{Fig: VSFFSF}.

\begin{challenge}
Audition candidates for a Grad-impersonator who can act on face scalar fields.
\end{challenge}

\subsubsection{Ebb} Ebb is our poor substitute for  divergence. It converts a CVF into VSF. Let $\vect{F}$ be a CVF on the combinatorial surface $(S,G)$. For a vertex $v \in \mc{V}$, let $E_+(v,\vect{F})$ be the set of edges with orientation given by $\vect{F}$ pointing out of $v$ and let $E_-(v,\vect{F})$ be the set of edges with orientation given by $\vect{F}$ pointing into $v$. Define the \defn{ebb} of $\vect{F}$ at a vertex $v$ to be:
\[
\ebb \vect{F}(v) = \sum_{e \in E_+(v,\vect{F})} \vect{F}(e) - \sum_{e \in E_-(v,\vect{F})} \vect{F}(e).
\]
Informally, $\ebb \vect{F}$ measures the net flow of cream pies out of a vertex. The ebb of the CVF from Figure \ref{Fig: CVF} is depicted in Figure \ref{Fig: Ebb Fig}. 
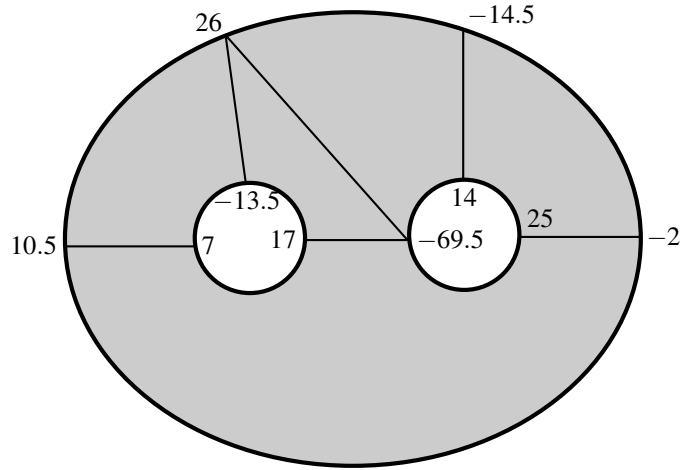
\begin{figure}[ht]
\labellist \small\hair 2pt 
\pinlabel {$10.5$} [r] at 0 191
\pinlabel {$26$} [br] at 140 371
\pinlabel {$-14.5$} [bl] at 344 377
\pinlabel {$-2$} [l] at 496 198
\pinlabel {$25$} [bl] at 395 203
\pinlabel {$-69.5$} [l] at 299 196
\pinlabel {$17$} [r] at 205 196
\pinlabel {$7$} [l] at 115 191
\pinlabel {$-13.5$} [t] at 158 242
\pinlabel {$14$} [t] at 344 245
\endlabellist 
\centering 
\includegraphics[scale=.45]{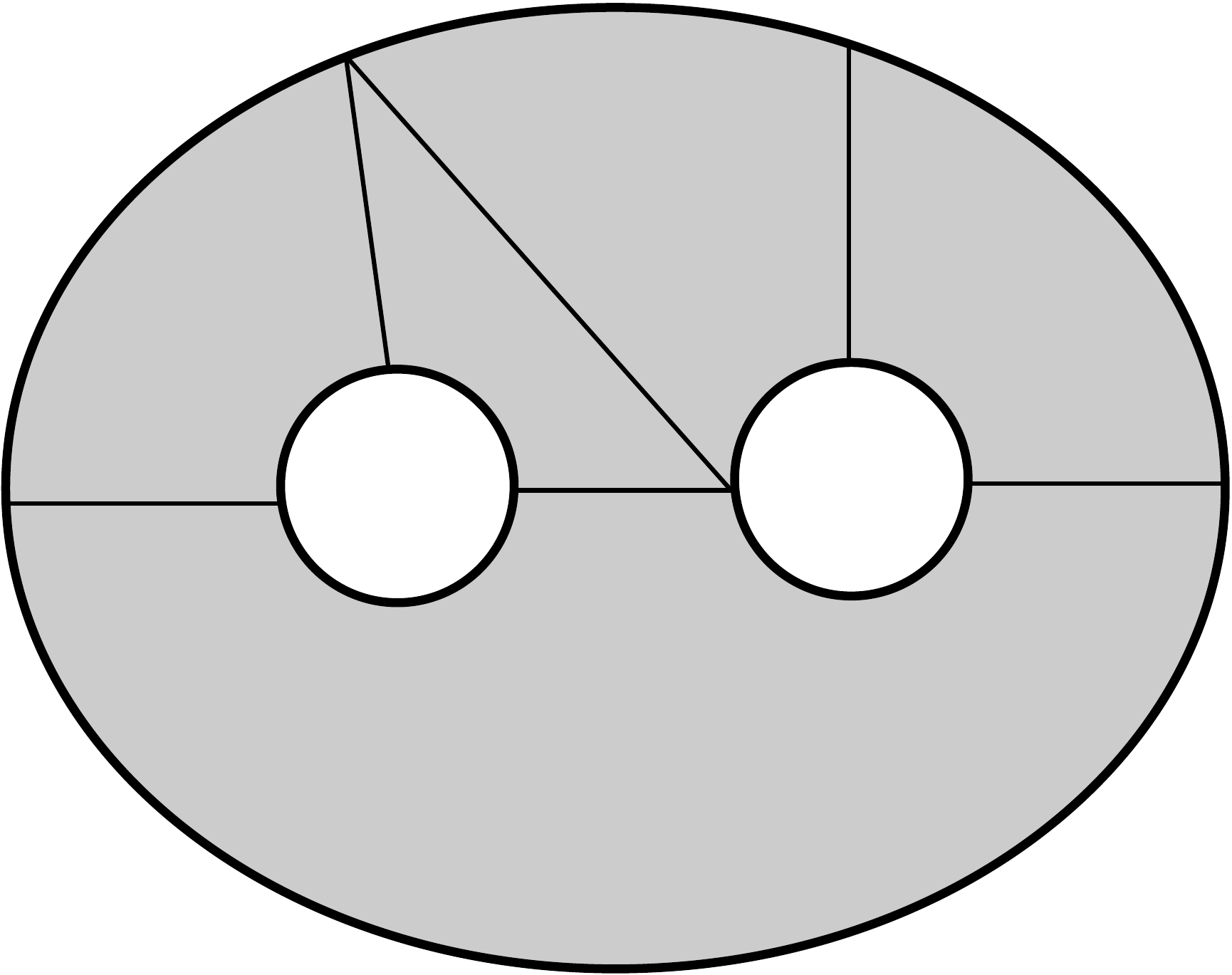}
\caption{The VSF pictured is the ebb of the CVF in Figure \ref{Fig: CVF}.}
\label{Fig: Ebb Fig}
\end{figure}

\begin{challenge}
Audition candidates for a Div-impersonator that produces a FSF.
\end{challenge}

\subsubsection{Whirl} Whirl is our mock scalar Curl. It converts a CVF into an FSF and measures the circulation of cream pies around the boundary of each oriented face. Let $\vect{F}$ be a CVF on $S$ and let $\sigma$ be a face with an orientation. As usual, we give the edges in $\boundary \sigma$ the orientation induced by the orientation of $\sigma$. Define
\[
\whirl \vect{F}(\sigma) = \cint_{\boundary \sigma} \vect{F}.
\]
For example, the whirl of the CVF in Figure \ref{Fig: CVF} assigns 0 to every face. Notice that if the orientation of $\sigma$ is reversed, then $\whirl \vect{F}$ changes sign. Also note that if $\vect{F} \same \vect{G}$ then $\whirl \vect{F} = \whirl \vect{G}$. If $(S,G)$ is an oriented surface, we assume that Whirl is defined using the orientation on each face of $G$ induced by the orientation of $(S,G)$. 

\subsection{Bad knockoffs of the original episodes.}

Each of our versions is acted on a combinatorial surface $(S,G)$. 

\begin{episode1a}\label{Episode 1A}
The following are equivalent for a CVF $\vect{F}$:
\begin{enumerate}
\item\label{embedded loop} For all \emph{embedded} edge loops $\phi$, $\cint_{\phi} \vect{F} = 0$.
\item\label{any loop} For \emph{all} edge loops $\phi$, $\cint_{\phi} \vect{F} = 0$.
\item\label{two paths} If $\phi$ and $\psi$ are two edge paths each joining a vertex $v$ to a vertex $w$ then $\cint_\phi \vect{F} = \cint_{\psi} \vect{G}$.
\item\label{making tilt} There is a VSF $f$ such that $\vect{F} \same \tilt f$.
\end{enumerate}
\end{episode1a}

\begin{proof}
Scene 1 of Episode 1' presents the obvious fact that \eqref{any loop} $\Rightarrow$ \eqref{embedded loop}. 

Scene 2, which shows that \eqref{embedded loop} $\Rightarrow$ \eqref{any loop}, is more subtle and shows the advantage of edge paths over merely piecewise C$^1$ paths. Assume, for a contradiction, that \eqref{embedded loop} holds but that \eqref{any loop} does not. Let 
\[\phi = e_1, \hdots, e_n \] 
be an edge loop in $(S,G)$ with the property that $\cint_\phi \vect{F} \neq 0$. There may be many such edge loops and they may contain different numbers of edges. We choose $\phi$ to be one which minimizes $n$. Let $v_i$ be the sink vertex for $e_i$. Since (2) holds, the edge loop $\phi$ is not embedded. Thus, either there are adjacent edges $e_j$ and $e_{j+1}$ in $\phi$ that are the same edge in $G$, or there are $j \neq k$ such that $v_j = v_k$. If the former happens, we can delete $e_j$ and $e_{j+1}$ from $\phi$ to obtain $\phi'$. Since $e_j$ and $e_{j+1}$ have opposite orientations in $\phi$, we have $\cint_{\phi'} \vect{F} = \cint_{\phi} \vect{F}$. But this contradicts the choice of $\phi$ to be a path of shortest length contradicting (3). Thus, there are vertices $v_j = v_k$ with $j \neq k$. We may choose $j$ and $k$ so that $j < k$ and so that $k - j$ is as small as possible. The edge path $\psi = e_{j+1}, \hdots, e_k$ is then an embedded edge loop in $G$ and so by \eqref{embedded loop}, $\cint_{\psi} \vect{F} = 0$. The path $\phi'= e_1, \hdots, e_j, e_{k+1}, \hdots, e_n$ is an edge loop with $\cint_{\phi'} \vect{F} = \cint_{\phi} \vect{F}$, but $\phi'$ is shorter than $\phi$ and so we have contradicted our choice of $\phi$. Figure \ref{Path Shortening} shows an example. Hence, \eqref{embedded loop} $\Rightarrow$ \eqref{any loop}.

\begin{figure}[ht]
\labellist \small\hair 2pt 
\pinlabel {$e_j$} [bl] at 95 293
\pinlabel {$v_j$} [tr] at 167 228
\pinlabel {$e_{j+1}$} [tr] at 283 185
\pinlabel {$e_k$} [bl] at 323 262
\pinlabel {$v_k$} [b] at 174 239
\pinlabel {$e_{k+1}$} [r] at 170 174
\endlabellist 
\centering 
\includegraphics[scale=.4]{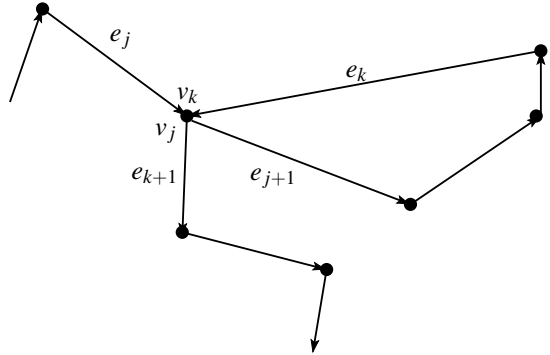}
\caption{The path $e_{j+1}, \hdots, e_k$ is an embedded edge loop and the path $\phi'= e_1, \hdots, e_j, e_{k+1}, \hdots, e_n$ is shorter than $\phi$}
\label{Path Shortening}
\end{figure}

Scene 3 happens very quickly. It shows that \eqref{any loop} $\Rightarrow$  \eqref{two paths}. Assume that \eqref{any loop} holds. Let $\phi = e_1, \hdots, e_n$ and $\psi = a_1, \hdots, a_m$ be two edge paths joining a vertex $v$ to a vertex $w$. Then the edge path
\[
\phi \ob{\psi}= e_1, \hdots, e_n, a_m, a_{m-1}, \hdots, a_1
\]
is an edge loop. Thus,
\[
0 = \cint_{\phi\ob{\psi}} \vect{F} =  \cint_{\phi} \vect{F} -\cint_{\psi} \vect{F},
\]
which implies \eqref{two paths}.

Scene 4 shows that \eqref{two paths} $\Rightarrow$ \eqref{any loop}; it unfolds even more quickly than Scene 3. Assume (4) and let $\phi$ be an edge loop in $G$ with $v$ the source vertex for the first edge in $\phi$. Let $e$ be an oriented edge having $v$ as its source and let $\ob{e}$ be the same edge but with the opposite orientation. Let $\psi = e, \ob{e}$. Clearly, $\cint_{\psi} \vect{F} = 0$. Since $\psi$ and $\phi$ both join $v$ to $v$, by \eqref{two paths} we have $\cint_{\phi}\vect{F} = 0$.

In Scene 5, the drama picks up a little. It shows that \eqref{making tilt} $\Rightarrow$ \eqref{any loop}. Recall that if $\vect{F} \same \vect{G}$ then for any edge path $\phi$, $\cint_{\phi} \vect{F} = \cint_{\phi} \vect{G}$. We may, therefore, assume that there is a VSF $f$ such that $\vect{F} = \tilt f$. Let $\phi = e_1, \hdots, e_n$ be an edge loop and let $v_i$ be the source vertex of $e_i$. Define $\epsilon_i = \epsilon(e_i, \tilt f)$.  By definition of the combinatorial integral, 
\[
\cint_\phi \tilt f = \epsilon_1\tilt f(e_1) + \epsilon_2\tilt f(e_2) + \hdots + \epsilon_n\tilt f(e_n).
\]

We observe that, for all $i$, 
\begin{equation}\label{epsilon id}
\epsilon_i\tilt f(e_i) = f(v_{i+1}) - f(v_i).
\end{equation}
To see this, recall that if $f(v_i) < f(v_{i+1})$ then $\tilt f$ gives $e_{i}$ the same orientation as that given by $\phi$ and $\tilt f(e_{i+1}) = f(v_{i+1}) - f(v_i)$.  If, on the other hand, $f(v_i) > f(v_{i+1})$ then $\tilt f$ gives $e_{i+1}$ the opposite orientation as that given by $\phi$ and $\tilt f(e_{i+1}) = f(v_i) - f(v_{i+1})$. If $f(v_i) = f(v_{i+1})$, then $\tilt f(e_{i+1}) = 0$. Thus,  in all cases, \eqref{epsilon id} holds. See Figure \ref{CombPath2}. Consequently,
\[\cint_\phi \tilt f = (f_{v_2} - f(v_1)) + \hdots + (f(v_{n-1}) - f(v_{n-2})) + (f(v_1) - f(v_{n-1})),
\]
which is clearly 0, as desired.

\begin{figure}[ht]
\labellist \small\hair 2pt 
\pinlabel {$v_1$} [t] at 143 85 
\pinlabel {$v_2$} [t] at 362 0 
\pinlabel {$v_3$} [l] at 490 92 
\pinlabel {$v_4$} [r] at 319 100
\pinlabel {$v_5$} [l] at 449 233
\pinlabel {$v_6$} [b] at 334 282 
\pinlabel {$v_7$} [t] at 222 157
\pinlabel {$v_8$} [b] at 186 279 
\pinlabel {$v_9$} [r] at 1 196
\pinlabel {$f(v_2) - f(v_1)$} [tr] at 290 32
\pinlabel {$f(v_3) - f(v_2)$} [tl] at 418 42
\pinlabel {$f(v_4) - f(v_3)$} [bl] at 367 101
\pinlabel {$f(v_5) - f(v_4)$} [l] at 392 173
\pinlabel {$f(v_6) - f(v_5)$} [bl] at 381 257
\pinlabel {$f(v_7) - f(v_6)$} [l] at 273 212
\pinlabel {$f(v_8) - f(v_7)$} [r] at 212 187
\pinlabel {$f(v_9) - f(v_8)$} [br] at 88 234
\pinlabel {$f(v_1) - f(v_9)$} [tr] at 71 145
\endlabellist 
\centering 
\includegraphics[scale=.45]{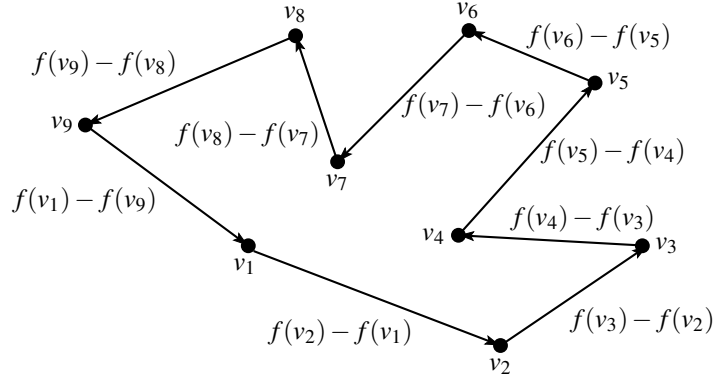}
\caption{The edge loop $\phi$. Each edge $e_i$ is labelled with $\epsilon_i \tilt f(e_i)$. Observe that all the labels on the edges sum to 0.}
\label{CombPath2}
\end{figure}

Scene 6, which shows that \eqref{two paths} $\Rightarrow$ \eqref{making tilt}, is where the pies really fly. The curtains part, revealing a CVF $\vect{F}$ satisfying \eqref{two paths}.  We wish to construct a VSF $f$ so that $\vect{F} \same \tilt f$. As in the classic version, we do this by integrating over paths. Without loss of generality, assume that $S$ is connected (if not, do the following in each component of $S$.) As in the classic version of the episode, all depends on choosing a home vertex $a$, on which the definition of $f$ is based. It matters not which vertex is chosen to be home, but different choices give different VSFs.

For each vertex $x \in \mc{V}$, we choose an edge path $\phi_x$ that begins at $a$ and ends at $x$ and define $f(x) = \cint_{\phi_x} \vect{F}$.  At first appearance, it seems that $f$ depends not only on $a$ but also on the chosen paths $\phi_x$. However, by \eqref{two paths}, the VSF depends only on $a$. To see that $\tilt f \same \vect{F}$, consider an edge $e$ with endpoints $v$ and $w$. Assume that $\vect{F}$ orients $e$ from $v$ to $w$. Let $\phi_v$ be a path from $a$ to $v$ and let $\phi_w$ be the path $\phi_v$ followed by the edge $e$, oriented from $v$ to $w$, as in Figure \ref{HomeVertex}.

\begin{figure}[ht]
\labellist \small\hair 2pt 
\pinlabel {$a$} [l] at 12 7
\pinlabel {$v$} [t] at 254 89
\pinlabel {$w$} [t] at 338 89
\pinlabel {$e$} [b] at 296 98
\endlabellist 
\centering 
\includegraphics[scale=.4]{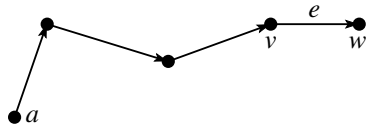}
\caption{The path $\phi_v$ runs from $a$ to $v$ and the path $\phi_w$ follows $\phi_v$ and then the edge $e$.}
\label{HomeVertex}
\end{figure}

Then, $\phi_w$ and $\vect{F}$ orient $e$ in the same direction and so 
\begin{equation}\label{tilted}
\cint_{\phi_w} \vect{F} - \cint_{\phi_v} \vect{F} = \vect{F}(e).
\end{equation}
If $\vect{F}(e) > 0$, then $\tilt f$ orients $e$ from $v$ to $w$, as does $\vect{F}$. In this case, by Equation \eqref{tilted}, $\tilt f$ and $\vect{F}$ give the same orientation and value to $e$. If, on the other hand, $\vect{F}(e) < 0$, then $\tilt f$ and $\vect{F}$ give opposite orientations and the values differ only in sign. With a final poke in the eye, we observe that if $\vect{F}(e) = 0$, then we also have $\tilt f (e) = 0$ and so $\tilt f \same \vect{F}$. \emph{Nyuk, nyuk, nyuk!}
\end{proof}

By the definition of whirl, we see immediately, without any unpleasant calculation, that if whirl clonks tilt, then nothing happens:

\begin{episode2a}[The TiltaWhirl Theorem]\label{Episode 2A}
If $f$ is a VSF on $(S,G)$ then $\whirl \tilt f = 0$.
\end{episode2a}

The next episode is our hack adaptation of Green's Theorem (\hyperref[Episode 3]{Episode 3}):

\begin{episode3a}[The Whirl Theorem]\label{Whirl Theorem}
If $\vect{F}$ is a CVF on an oriented combinatorial surface $(S,G)$, then 
\[
\cint_{\boundary S} \vect{F} = \cint_{S} \whirl \vect{F}.
\]
\end{episode3a}

\begin{proof}
Our coarse actors really shine in this episode, for our result doesn't rely on any unilluminating calculations, obscure definitions, or subtle properties of C$^1$ curves. The episode opens with a CVF $\vect{F}$ on an oriented combinatorial surface $(S,G)$. In saunters our hero, Whirl. By definition, for $\sigma \in \mc{F}$,
\[
\whirl \vect{F} (\sigma) = \sum_{e \subset \boundary \sigma} \epsilon(e,\vect{F})\vect{F}(e),
\]
where the sum is over all edges $e$ in the boundary of $\sigma$. Each of those edges has an orientation induced by the orientation of $S$. By the definition of $G$, in the sum
\begin{equation}\label{whirl expansion}
\cint_\mc{F} \whirl \vect{F} = \sum_{\sigma \in \mc{F}} \whirl \vect{F}(\sigma) = \sum_{\sigma \in \mc{F}} \sum_{e \subset \boundary \sigma} \epsilon(e,\vect{F}) \vect{F}(e)
\end{equation}
each edge $e \in \mc{E}$ appears once or twice. It appears once exactly when $e \subset \boundary S$ and it appears twice when $e \not\subset \boundary S$. If an edge $e$ appears twice, it is shared by the boundaries of distinct faces $\sigma$ and $\tau$. Since $(S,G)$ is oriented, if two faces $\sigma$ and $\tau$ are adjacent along an edge $e$, they induce opposite orientations on $e$.  Thus in Equation \eqref{whirl expansion}, all the terms cancel except for those coming from the edges in $\boundary S$. Thus, the final term in Equation \eqref{whirl expansion} is equal to $\sum_{e \subset \boundary S} \epsilon(e, \vect{F})\vect{F}(e).$ But this is exactly the definition of $\cint_{\boundary S} \vect{F}$. \emph{Nyuk, nyuk, nyuk!}
\end{proof}

\section{Refining the stand-in stooges}\label{Section: Refining}
Having appreciated the coarse acting by our stooge-wannabes Tilt, Ebb, and Whirl, we might now hope that we can improve our episodes by improving the scenery in them. Indeed we will show that if we refine our scenery enough, then Whirl and Curl are indistinguishable. Furthermore, Theorem \ref{Green's for V/H} shows that \hyperref[Episode 3]{Episode 3} and \hyperref[Whirl Theorem]{Episode 3'} also become indistinguishable.

\begin{challenge}
After reading this section, show that Ebb and Div are indistinguishable, in the same sense that Whirl and Curl are. What is the relationship between Tilt and Grad?
\end{challenge}

Throughout this section, let $S \subset \R^2$ be a compact surface with piecewise C$^1$ boundary. Let $\vect{F}$ be a C$^1$ vector field on $S$. Suppose that $G \subset S$ is a graph so that $(S,G)$ is a combinatorial surface. We begin by constructing a CVF $\vect{F}_G$ induced by $\vect{F}$. For an edge $e \in \mc{E}$, we let $\vect{F}_G$ give $e$ an orientation so that $\vect{F}_G(e) = \int_e \vect{F}\vl$ is non-negative.  

We observe:
\begin{lemma}\label{CVF VF equiv}
Suppose that $e$ is an oriented edge of $G$. Then $\cint_e \vect{F}_G = \int_e \vect{F}\vl$.
\end{lemma}
\begin{proof}
Let $\vect{G}$ be a CVF such that $\vect{G} \same \vect{F}$ and so that $\vect{G}$ gives $e$ the same orientation as the given orientation of $e$. Then $\cint_e \vect{G} = \int_e \vect{F}\vl$. Since $\cint_e \vect{F}_G = \cint_e \vect{G}$, we have our result.
\end{proof}

\subsection{Whirl and Curl}

The fundamental relationship between Whirl and Curl arises from applications of the Mean Value Theorems for Integrals and Derivatives. 

\begin{theorem}[MVT for vector fields on rectangles]\label{MVTR}
Suppose that $\vect{F} = \begin{pmatrix}M \\ N \end{pmatrix}$ is a differentiable vector field defined on a solid rectangle $R \subset \R^2$ of positive area with sides parallel to the $x$ and $y$ axes.  Then there exist points $\vect{x}, \vect{y} \in R$ such that
\[
\frac{1}{\area(R)}\int_{\boundary R} \vect{F} \cdot d\vect{s} = \llpartial{N}{y}(\vect{y}) - \llpartial{M}{x}(\vect{x}).
\]
\end{theorem}
\begin{proof}
Suppose that $R = [a,b] \times [c,d]$. Parameterize the top and bottom sides of $\boundary R$ as $(t,c)$ and $(t,d)$ for $a \leq t \leq b$, respectively. Parameterize the left and right sides of $\boundary R$ as $(a, t)$ and $(b, t)$ for $c \leq t \leq d$ respectively. Note that our parameterizations of the top and left sides of $\boundary R$ have the opposite orientations from that induced by $\boundary R$. We have (by Definition \eqref{Defn vl})
\begin{equation}\label{Eqn MVTI}
\frac{1}{\area(R)} \int_{\boundary R} \vect{F} \cdot d\vect{s} = \frac{-1}{b-a}\int_a^b \frac{M(t,d) - M(t,c)}{d-c}\thinspace dt + \frac{1}{d-c}\int_c^d\frac{N(b,t) - N(a,t)}{b-a} \thinspace dt 
\end{equation}

Since $M$ and $N$ are continuous, the integrands are continuous. By the Mean Value Theorem for Integrals, there exists $(x_0,y_0) \in R$ so that the right side of  Equation \eqref{Eqn MVTI} equals
\begin{equation}\label{post MVTI}
-\frac{M(x_0,d) - M(x_0,c)}{d-c} + \frac{N(b,y_0) - N(a,y_0)}{b-a}.
\end{equation}
Since the functions $M(x_0,\cdot)$ and $N(\cdot, y_0)$ are differentiable on the intervals $[c,d]$ and $[a,b]$ respectively, by the Mean Value Theorem for Derivatives, there exists $(x_1,y_1) \in R$ so that Expression \eqref{post MVTI} equals
\[
-\llpartial{M}{y}(x_0,y_1) + \llpartial{N}{x}(x_1,y_0).
\]
Letting $\vect{x} = (x_0,y_1)$ and $\vect{y} = (x_1,y_0)$, our show comes to its rousing conclusion.
\end{proof}

As an immediate corollary, observe:
\begin{corollary}\label{Circulation and Curl}
Let $R$ be a solid rectangle with sides parallel to the $x$ and $y$ axes. Let $G$ be the graph in $R$ having vertices only at the corners of $\boundary R$. Let $\vect{F} = \begin{pmatrix} M \\ N \end{pmatrix}$ be a differentiable vector field on $R$ such that either $M = 0$ or $N = 0$. Then there exists $\vect{x} \in R$ such that
\[
\frac{1}{\area(R)} \whirl \vect{F}_G (R) = \scurl \vect{F}(\vect{x}).
\]
\end{corollary}

Similarly, Theorem \ref{MVTR} will tell us that the limit of Whirl on rectangles is Curl. This gives a rigorous proof of Equation \eqref{Curl Revealed} for the case when the curves $\gamma_n$ are the boundaries of rectangles, without the use of \hyperref[Episode 3]{Green's theorem}.

\begin{corollary}
Suppose that  $R_n$ is a sequence of rectangles in $\R^2$, each with sides parallel to the $x$ and $y$ axes. Suppose that as $n \to \infty$, the rectangles $R_n$ converge to a point $\vect{x} \in \R^2$ and that $\vect{F}$ is a C$^1$ vector field defined on $\bigcup_n R_n$. Then
\[
\lim_{n \to \infty} \frac{1}{\area(R_n)} \int_{\boundary R_n} \vect{F} \vl = \scurl \vect{F}(\vect{x}).
\]
\end{corollary}
Observe that the integral on the left hand side of the equation is the whirl of $\vect{F}_{G_n}$ where $G_n$ is the graph with vertices at the corners of $R_n$.
\begin{proof}
By Theorem \ref{MVTR}, there exist $\vect{x}_n, \vect{y}_n \in R_n$ such that
\[
\frac{1}{\area(R_n)}\int_{\boundary R_n} \vect{F} \vl = \llpartial{N}{y}(\vect{y}_n) - \llpartial{M}{x}(\vect{x}_n).
\]
Since $\vect{F}$ is C$^1$, both $\llpartial{N}{y}$ and $\llpartial{M}{x}$ are continuous. Since both $(\vect{y}_n)$ and $(\vect{x}_n)$ converge to $\vect{x}$, the quantity $\llpartial{N}{y}(\vect{y}_n) - \llpartial{M}{x}(\vect{x}_n)$ converges to $\scurl \vect{F}(\vect{x})$ as desired.
\end{proof}

We are very close to obtaining Green's theorem for certain types of surfaces. To be precise, we say that a compact surface $S \subset \R^2$ and a vector field $\vect{F}$ on $S$ \defn{satisfy Green's theorem} if
\[
\int_{\boundary S} \vect{F} \vl = \iint_{S} \scurl \vect{F} \thinspace dA.
\]

With what we have so far, it is easy to show:
\begin{theorem}\label{Green's for V/H}
If $S \subset \R^2$ is a compact surface with boundary that is the union of finitely many horizontal and vertical line segments (we call such a surface an \defn{V/H surface}), then $S$ and any C$^1$ vector field $\vect{F} = \begin{pmatrix} M \\ N \end{pmatrix}$ on $S$ satisfy Green's theorem.
\end{theorem}

\begin{proof}
For each $n$, let $G_n$ be a graph in $S$, such that:
\begin{itemize}
\item $(S,G_n)$ is a combinatorial surface.
\item Each edge of $G_n$ is either a horizontal or a vertical line segment.
\item Every face of $G_n$ is a rectangle.
\item For all $n$, $G_n \subset G_{n+1}$.
\item As $n \to \infty$, the maximal diameter of a face of $G_n$ converges to 0.
\end{itemize}

Let $\vect{F}^\vect{i} = \begin{pmatrix} M \\ 0 \end{pmatrix}$ and $\vect{F}^\vect{j} = \begin{pmatrix} 0 \\ N \end{pmatrix}$ and let $\vect{F}^*$ be either of $\vect{F}^\vect{i}$ or $\vect{F}^\vect{j}$. Define $\vect{F}_n = \vect{F}^*_{G_n}$.

By Lemma \ref{CVF VF equiv}, we have 
\begin{equation}\label{defn of F_n}
\int_{\boundary S} \vect{F} \vl = \cint_{\boundary S} \vect{F}_n.
\end{equation} By the Whirl Theorem,
\begin{equation}\label{apply whirl thm}
\cint_{\boundary S} \vect{F}_n = \ciint_{S} \whirl \vect{F}_n.
\end{equation}

By definition, 
\begin{equation}\label{apply defn of int}
\ciint_S \whirl \vect{F}_n = \sum_{i=1}^{m_n} \whirl \vect{F}_n(\sigma^n_i)
\end{equation}
 where $\sigma^n_1, \hdots, \sigma^n_{m_n}$ are the faces of $G_n$. Since those faces are all rectangles, by Lemma \ref{Circulation and Curl}, there exists $\vect{x}^n_i \in \sigma^n_i$ such that $\whirl \vect{F}_n(i) = \scurl \vect{F}(\vect{x}^n_i)(\area (\sigma^n_i))$. See Figure \ref{Fig: Riemann} for an example. 

\begin{figure}[ht]
\includegraphics[scale=0.4]{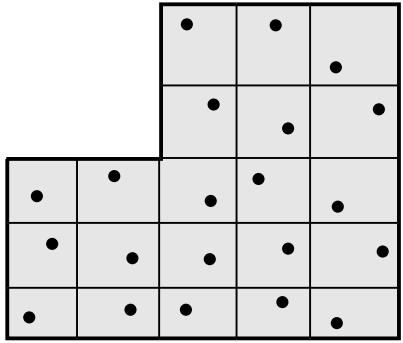}
\caption{A combinatorial surface $(S,G_n)$ with the points $\vect{x}^n_i$ marked in each face $\sigma^n_i$}.
\label{Fig: Riemann}
\end{figure}
Combining Equations \eqref{defn of F_n}, \eqref{apply whirl thm}, and \eqref{apply defn of int}:
\[
\int_{\boundary S} \vect{F} \vl = \sum_{i=1}^{m_n} \scurl \vect{F}(\vect{x}^n_i)(\area (\sigma^n_i)).
\]
But the right-hand side is a Riemann sum of a continuous function, and the left-hand side is constant in $n$ and so, taking the limit as $n \to \infty$, we conclude that $S$ and $\vect{F}^{\vect{i}}$ and $S$ and $\vect{F}^\vect{j}$ satisfy Green's theorem.

Since 
\[\begin{array}{rcl}
\int_{\boundary S} \vect{F} \vl &=& \int_{\boundary S} \vect{F}^\vect{i} \vl + \int_{\boundary S} \vect{F}^\vect{j} \vl\text{ , and } \\
\iint_S \scurl \vect{F}\thinspace dA &=& \iint_S \scurl \vect{F}^\vect{i}\thinspace dA + \iint_S \scurl \vect{F}^\vect{j}\thinspace dA 
\end{array}\]
the surface $S$ and the vector field $\vect{F}$ satisfy Green's theorem.
\end{proof}

In the next section, we perform a play-within-the-play to show that we can obtain Green's Theorem for surfaces other than V/H-surfaces.

\subsection{Bending the scenery}
In this section we rely heavily on the elementary linear algebra of $2\times2$ matrices. We denote the transpose of a matrix $B$ by $B^T$ and the derivative of a function $H$ at a point $a$ by $DH(a)$. In our context, $DH(a)$ will always be a 2$\times$2 matrix. 

For reasons that will become evident we define the \defn{scalar curl} of  a 2$\times$2 matrix to be
\[
\scurl \begin{pmatrix}a & b \\ c & d \end{pmatrix} = c - b
\]
An easy computation shows that, for 2$\times$2 matrices $A$ and $B$,
\begin{equation}\label{congruence}
\scurl (B^T A B) = (\scurl A)\det B.
\end{equation}
If $\vect{F}\co S \to \R^2$ is a vector field on a surface $S \subset \R^2$ we observe by Definition \eqref{Disguised Curl} that at $x \in S$
\[
\scurl \vect{F}(x) = \scurl D\vect{F}(\vect{a}).
\]

Using these simple definitions, we can study the relationship between surfaces related by certain kinds of distortions. Let $\wihat{S}$ and $S$ be compact connected surfaces in $\R^2$, each bounded by piecewise C$^1$ curves. Suppose that $H\co \wihat{S} \to S$ is a continuous bijection of class C$^2$ (i.e. all second partial derivatives exist and are continuous) and with the property that $\det DH$ is non-zero on $\wihat{S}$. We call $H$ a \defn{diffeomorphism} from $\wihat{S}$ to $S$. Since $H$ is C$^2$, the entries of $DH$ are continuous and $\det DH$ is never 0 on the interior of $S$. Thus, the sign $\epsilon$ of $\det DH$ is either always positive or always negative. If it is the former, we say that $H$ is \defn{orientation-preserving}; and if the latter, that $H$ is \defn{orientation-reversing}. 

\begin{example}\label{Not Type III}
Let $f,g\co \R \to \R$ be C$^2$ functions with the property that $f(x) < g(x)$ for all $x$. Let $a < b$ and $c < d$ be real numbers. Define
\[
H(x,y) = \Big(x, \frac{g(x) - f(x)}{d-c}(y - c) + f(x)\Big).
\]
Then $H \co \R^2 \to \R^2$ is a C$^2$ function with the property that $\det DH(x,y) > 0$ for all $(x,y)$. Let $\wihat{S} = [a,b] \times [c,d]$ and let $S = H(\wihat{S})$. Then the restriction of $H$ to $\wihat{S}$ is an orientation preserving diffeomorphism from the rectangle $\wihat{S}$ to $S$.

If we choose \[f(x) = \Big\{\begin{array}{lr} x^3\sin(1/x) \hspace{.2in} & x \neq 0 \\ 0 & x = 0 \end{array}\] and $g(x) = f(x) + 1$, then $H$ is a diffeomorphism from the square $\wihat{S} = [0,1]\times[0,1]$ to a region whose boundary has infinitely many oscillations (Figure \ref{Strange Region}). This region cannot be subdivided into finitely many Type III regions.
\end{example}

\begin{figure}[ht]
\labellist \small\hair 2pt 
\pinlabel {$\wihat{S}$} at 96 143 
\pinlabel {$S$} at 377 134
\pinlabel {$H$} [b] at 250 144 
\endlabellist 
\centering
\includegraphics[scale=0.4]{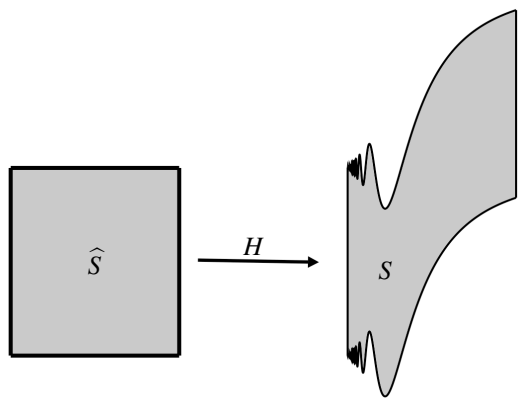}
\caption{The square $\wihat{S}$ is on the left and the distorted surface $S$ is on the right. The oscillations in $\boundary S$ have been exaggerated for effect.}
\label{Strange Region}
\end{figure}

Given a C$^1$ vector field $\vect{F}$ on $S$, we can ``pull'' it back to a C$^1$ vector field $\wihat{F}$ on $\wihat{S}$ defined by
\[
\wihat{\vect{F}} = DH^T(x)\vect{F}(H(x))\]
for all $x \in \wihat{S}$. The vector field $\wihat{F}$ is essentially the vector field $\vect{F}$ moved to the surface $\wihat{S}$ and adjusted to account for the distortion caused by $DH$.

The next theorem shows that integrals of the two vector fields around the boundaries of the surfaces are equal and their scalar curls are related in a particularly simple way.

\begin{theorem}\label{Thm: COV}
Let $\wihat{\gamma}\co [a,b] \to S$ be a C$^1$ curve, and let $\gamma = H\circ \wihat{\gamma}$. Then
\[\begin{array}{rcl}
\int\limits_{\boundary S} \vect{F} \vl &=&\epsilon \int\limits_{\boundary \wihat{S}} \wihat{F} \vl, \text{ and} \\
\scurl \wihat{\vect{F}}(x) &=&\epsilon \scurl \vect{F}(H(x))|\det DH(x)| \\
\end{array}
\]
with the last equality holding for all $x \in \wihat{S}$. (Recall that $\epsilon$ is the sign of $\det DH$.)
\end{theorem}
\begin{proof}
By Definition \eqref{Defn vl}, we have
\[
\int_{\wihat{\gamma}} \wihat{F}\vl = \int_a^b \Big(DH^T(\gamma(t))\vect{F}(\gamma(t))\Big)\cdot \wihat{\gamma}\thinspace'(t) \thinspace dt.
\]
By the chain rule  \cite[Theorem 8.15]{Browder}, 
\[
\wihat{\gamma}\thinspace'(t) = \frac{d}{dt} H^{-1}(\gamma(t)) = DH^{-1}(\gamma(t))\gamma'(t).
\]
Since for any two vectors $\vect{v}, \vect{w} \in \R^2$, $\vect{v} \cdot \vect{w} = \vect{v}^T\vect{w}$ and since for any two 2$\times$2 matrices $A,B$, $(AB)^T = B^TA^T$, we have
\[
\Big(DH^T(\gamma(t))\vect{F}(\gamma(t))\Big)\cdot \wihat{\gamma}'(t) =  \vect{F}(\gamma(t))\cdot \gamma'(t).
\]
Consequently,
\[
\int_{\wihat{\gamma}} \wihat{F}\vl = \int_{\gamma}\vect{F}\vl,
\]
which is the first conclusion of the theorem.

To obtain the second conclusion, let $H(x,y) = (H_1(x,y), H_2(x,y))$ and let 
\[
Q_1 = \begin{pmatrix} \frac{\partial^2 H_1}{\partial x^2} & \frac{\partial^2 H_1}{\partial y \partial x} \\
				\frac{\partial^2 H_1}{\partial x\partial y} & \frac{\partial^2 H_1}{\partial y ^2}
	\end{pmatrix}
\hspace{.5in} \text{ and } \hspace{.5in}
Q_2 = \begin{pmatrix} \frac{\partial^2 H_2}{\partial x^2} & \frac{\partial^2 H_2}{\partial y \partial x} \\
				\frac{\partial^2 H_2}{\partial x\partial y} & \frac{\partial^2 H_2}{\partial y ^2}
	\end{pmatrix}.
\]
Notice that by the equality of mixed second partial derivatives \cite[Theorem 8.24]{Browder}, $\scurl Q_1 = \scurl Q_2 = 0$.

A computation shows that $D\wihat{\vect{F}}$ is equal to
\[
M(H)Q_1 + N(H)Q_2 + \big(DH^T\big)D(\vect{F}(H))
\]
The scalar curl for matrices is linear and so
\[
\scurl \wihat{\vect{F}} = \scurl \big(DH^TD(\vect{F}(H))\big) = \scurl \big(DH^TD\vect{F}(H)DH\big),
\]
where we've used the chain rule to obtain the second equality. Thus, by Equation \eqref{congruence}, $\scurl \wihat{\vect{F}} = \big(\scurl \vect{F}\big) \det DH$ and so
\[
\scurl \wihat{\vect{F}} = \epsilon \big(\scurl \vect{F}\big) \big|\det DH\big|,
\]
as desired.
\end{proof}

\begin{corollary}\label{Cor: Green's Satisfied}
Let $\wihat{S},S \subset \R^2$ be compact surfaces with piecewise C$^1$ boundaries and let $H\co \wihat{S} \to S$ be a diffeomorphism. Let $\vect{F}$ be a vector field on $S$ and let $\wihat{\vect{F}} = DH^T\vect{F}(H)$. Then $S$ and $\vect{F}$ satisfy Green's theorem if and only if $\wihat{S}$ and $\wihat{\vect{F}}$ do.
\end{corollary}

\begin{proof}
Without loss of generality, we may assume that $S$ and $\wihat{S}$ are connected (if not, do what follows for each component). If $\gamma\co [a,b] \to \R^2$ is a C$^1$ parameterization of a portion of $\boundary \wihat{S}$, then $H(\gamma)$ is a  C$^1$ parameterization of a portion of $\boundary S$, possibly with the wrong orientation. If $H$ is orientation-preserving, then $H(\gamma)$ has the same orientation as that of $\boundary S$; otherwise, it has the opposite orientation. Applying Theorem \ref{Thm: COV} to the C$^1$ portions of $\boundary S$, we see that
\begin{equation}\label{Line integrals equal}
\int_{\boundary \wihat{S}} \wihat{\vect{F}} \vl = \epsilon \int_{\boundary S} \vect{F} \vl.
\end{equation}

We now turn to double integrals over $S$ and $\wihat{S}$. By the change of variables theorem \cite[Theorem 17.1]{Munkres} (applied to the interiors of $S$ and $\wihat{S}$), for any C$^1$ function $f\co S \to \R$:
\[
\iint_{S} f \thinspace dA = \iint_{\wihat{S}} f\circ H \thinspace |\det DH| \thinspace dA
\]
Letting $f = \scurl \vect{F}$ and applying Theorem \ref{Thm: COV}, we obtain:

\[
\epsilon \iint_S \scurl \vect{F} \thinspace dA =  \iint_{\wihat{S}} \scurl \vect{\wihat{F}} \thinspace dA.
\]

Thus, the surface $S$ and vector field $\vect{F}$ satisfy Green's theorem if and only if the surface $\wihat{S}$ and vector field $\wihat{F}$ do.
\end{proof}

\begin{corollary}\label{Porting Green's}
Suppose that $\wihat{S} \subset \R^2$ is a V/H surface and that there is a diffeomorphism $H$ of $\wihat{S}$ onto a surface $S \subset \R^2$. Then $S$ and $\vect{F}$ satisfy Green's theorem for any C$^1$ vector field $\vect{F}$ on $S$.
\end{corollary}
\begin{proof}
Since $H$ is C$^2$, the vector field $\wihat{F} = DH^T \vect{F}(H)$ is of class C$^1$. We then apply Corollary \ref{Cor: Green's Satisfied} and Theorem \ref{Green's for V/H}.
\end{proof}

By the classification of surfaces up to piecewise C$^2$ diffeomorphism, it follows that Green's theorem holds for all compact surfaces $S \subset \R^2$ with piecewise C$^2$ boundary (i.e. surfaces with boundary having parameterizations with continuously differentiable and non-vanishing first derivatives.) Making this precise would show us that  the coarse \hyperref[Whirl Theorem]{Episode 3'} can be improved to obtain, in the limit, the classic \hyperref[Episode 3]{Episode 3}. Rather than carrying a player piano up those stairs, however, we content ourselves with observing that our result is better than the traditional result that Type III surfaces and C$^1$ vector fields satisfy Green's theorem\footnote{Of course, we have not improved on Apostol's result \cite[Theorem 10.43]{Apostol}.}.

\begin{example}
Let $S$ be the surface 
\[
S = \{(x,y) \in \R^2 : 0 \leq x \leq 1, g(x) \leq y \leq g(x) + 1 \},
\]
where $g\co[0,1] \to \R$ is the function 
\[
g(x) = \Big\{ \begin{array}{lr} x^3 \sin(1/x) \hspace{.2in} & x \neq 0 \\ 0 & x = 0. \end{array}
\]
Then, by example \ref{Not Type III}, the surface $S$ and every C$^1$ vector field $\vect{F}$ satisfy Green's theorem.
\end{example}

\begin{challenge}
Use Theorem \ref{Porting Green's} (and maybe some other tricks) to show that the unit disc in $\R^2$ together with any C$^1$ vector field on it satisfy Green's theorem.
\end{challenge}

\section{The Season Finale}\label{Cohomology}
In this final section of the viewer's guide we explain how both the coarse actors and the original 3 stooges fit into the over-arching structure known as ``cohomology theory''. In fact, we will describe two cohomology theories, which can be thought of as the labor unions for the coarse actors and the original 3 stooges.

\subsection{The Coarse Screen Actor's Guild}
Let $(S,G)$ be an oriented combinatorial surface. The set $C^0(S,G)$ of VSFs is a finite dimensional real vector space (with dimension equal to the number of vertices). The set $C^2(S,G)$ of FSFs is also a finite dimensional real vector space (with dimension equal to the number of faces). In both cases, the vector space operations are the usual scaling and addition of real-valued functions.

If we fix an orientation on each edge of $e$, the set $C^1(S,G)$ of \emph{canonical} CVFs is also a real vector space\footnote{If you don't like working with only canonical CVFs, it is instead possible to work with the set of equivalence classes of CVFs under the equivalence relation $\same$.}, as follows. Suppose that $k \in \R$ and that $\vect{F}$ is a canonical CVF. Let $k\vect{F}$ be the CVF where each edge has the same orientation as that given by $\vect{F}$ and $(k\vect{F})(e) = k\vect{F}(e)$ for each edge $e \in \mc{E}$. For canonical CVFs $\vect{F}$ and $\vect{G}$ and an edge $e$ we let $\vect{F} + \vect{G}$ give $e$ the same orientation as that given by $\vect{F}$ and $\vect{G}$ and we define $(\vect{F} + \vect{G})(e) = \vect{F}(e) + \vect{G}(e)$. Note that $k\vect{F}$ and $\vect{F} + \vect{G}$ are canonical CVFs. Then $C^1(S,G)$ is a real vector space with dimension equal to the number of edges in $G$.

The vector space $C^i(S,G)$ is known as the $i$th \defn{cochain group} of $(S,G)$. It depends on both $S$ and $G$. However, the quantity,  known as the \defn{euler characteristic} of $S$, equal to $\dim C^0(S) - \dim C^1(S) + \dim C^2(S)$ is independent of $G$.  It is the primary example of what is known as a ``topological invariant'' of $S$.  We would like to turn the cochain groups themselves into topological invariants. The resulting vector spaces are called the ``cohomology groups'' of $S$. To explain how, we begin with a brief digression to the world of quotient vector spaces.

Whenever we have a vector space $V$ and a subspace $W$ we can form a new vector space $V/W$ called the \defn{quotient vector space} as follows. We declare $v_1, v_2 \in V$ to be equivalent, if $v_1 - v_2 \in W$. That is, $v$ and $w$ are ``the same''\footnote{If you haven't encountered equivalence relations before, you may like to compare this to how we work with angles: Two angles $\alpha$ and $\beta$ are the same angle if and only if $\alpha - \beta \in \{\hdots, -4\pi, -2\pi, 0, 2\pi, 4\pi, \hdots\}$.} if they differ by an element of $W$. The set $V/W$ is the set of equivalence classes and the vector space operations on $V$ produce well-defined vector space operations on $V/W$. For example, if $v \in V$, we let $[v]$ be the set of all vectors who differ from $v$ by an element of $W$ and we define $[u] + [v] = [u + v]$ for $u,v \in V$ and $k[v] = [kv]$ for $k \in \R$ and $v \in V$.

For an oriented combinatorial surface $(S,G)$ where every edge in $\mc{E}$ has also been given an orientation, we observe that $\whirl\co C^1(S,G) \to C^2(S,G)$ is a linear map. Since the image of Tilt lies in the kernel of Whirl (by the \hyperref[Episode 2A]{TiltaWhirl Theorem}), we might try to form the quotient vector space $\ker \whirl/\im \tilt$. Unfortunately, though, the vector fields produced by Tilt may not be canonical and, thus, might not be elements of $C^1(S,G)$. To fix this, we alter the definition of Tilt, to produce a similar operator, which we call $\ob{\tilt}$.

For a VSF $f\co \mc{V} \to \R$, and an oriented edge $e \in \mc{E}$ with source vertex $v$ and sink vertex $w$, define
\[
\ob{\tilt} f(e) = f(w) - f(v)
\]
and let $\ob{\tilt} f$ give $e$ the fixed orientation given to $e$ at the outset. Then $\ob{\tilt} f$ is a canonical CVF and, in fact, it is the unique canonical CVF that is the same as $\tilt f$. The map $\ob{\tilt} \co C^0(S,G) \to C^1(S,G)$ is linear and (by the \hyperref[Episode 2A]{TiltaWhirl Theorem}) the image of $\ob{\tilt}$ is a subset of the kernel of $\whirl$. Thus, our coarse actors $\ob{\tilt}$ and $\whirl$ are linear maps between cochain groups:
\[
C^0(S,G) \stackrel{\ob{\tilt}}{\to} C^1(S,G) \stackrel{\whirl}{\to} C^2(S,G).
\]
We define:
\[\begin{array}{rcl}
 H^0(S) &=& \ker \thinspace \ob{\tilt} \\
 H^1(S) &=& (\ker \thinspace \whirl)/(\im \thinspace \ob{\tilt})\\
 H^2(S) &=& \im\thinspace \whirl.
 \end{array}\]
 The vector space $H^i(S)$ is called the $i$th cohomology group of $S$. It turns out that, up to vector space isomorphism, it does not depend on $G$. 
 \begin{example}[Poincar\'e's Theorem]
 Let $D$ be the closed unit disc in $\R^2$. Then $H^1(D) = 0$. 
 
To see this, let $G \subset D$ be the graph having one vertex $v$ on $\boundary D$, one edge $e$ on $\boundary D$ and one face $\sigma$. Orient the edge counter-clockwise, as in Figure \ref{Disc}. To show that $H^1(D) = 0$, we must show that $\ker \whirl\subset \im \ob{\tilt}$. To that end, let $\vect{F}$ be a CVF on $(D,G)$ such that $\whirl \vect{F} = 0$. Since $G$ has a single face and since the edge $e$ is the boundary of that face, we must have 
\[
\vect{F}(e) = \cint_e \vect{F} = \whirl \vect{F}(\sigma) = 0.
\]
Defining $f(v) = 0$, we obviously have $\vect{F} = \ob{\tilt} f$, as desired. 
\end{example}

\begin{figure}[ht]
\labellist \small\hair 2pt 
\pinlabel {$v$} [b] at 174 352
\pinlabel {$e$} [bl] at 281 310
\pinlabel {$\sigma$} at 174 174 
\endlabellist 
\centering 
\includegraphics[scale=0.4]{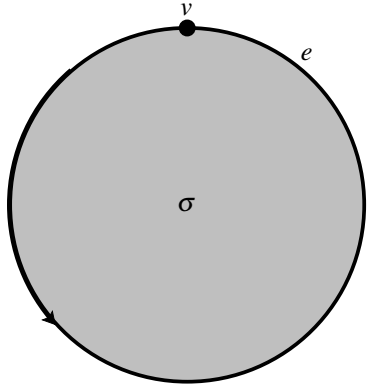}
\caption{The disc $D$ together with a graph $G$ having one vertex, one edge, and one face.}
\label{Disc}
\end{figure}

A moderately more difficult result is:

\begin{example}\label{annulus}
If $S = \{x \in \R^2 : 1 \leq ||x|| \leq 2|\}$, then $\dim H^1(S) \neq 0$. 

To see this, choose $G \subset S$ to be the graph with two vertices on each boundary component of $S$ and two edges in the interior of $S$. Fix orientations on the edges of $G$ as in Figure \ref{Fig: Annulus}.  Let $\vect{F}\co \mc{E} \to \R$ assign 1 to an edge on the inner boundary component, -1 to an edge on the outer boundary component  and 0 to the other edges as in Figure \ref{Fig: Annulus} (and give the edges the same orientations as in Figure \ref{Fig: Annulus}.) It is easy to verify that $\whirl \vect{F} = 0$. Let $\alpha$ be the inner boundary component of $S$, and notice that $\int_{\alpha} \vect{F} = 1$. Thus, by \hyperref[Episode 1A]{Episode 1'}, $\vect{F}$ cannot be in the image of $\ob{\tilt}$. Hence, $H^1(S) \neq 0$.
\end{example}

\begin{figure}[ht]
\labellist \small\hair 2pt 
\pinlabel {$0$} [br] at 89 316
\pinlabel {$0$} [bl] at 189 228
\pinlabel {$0$} [b] at 73 164
\pinlabel {$0$} [b] at 257 164
\pinlabel {$1$} [br] at 178 112
\pinlabel {$-1$} [br] at 201 14
\endlabellist 
\centering 
\includegraphics[scale=.5]{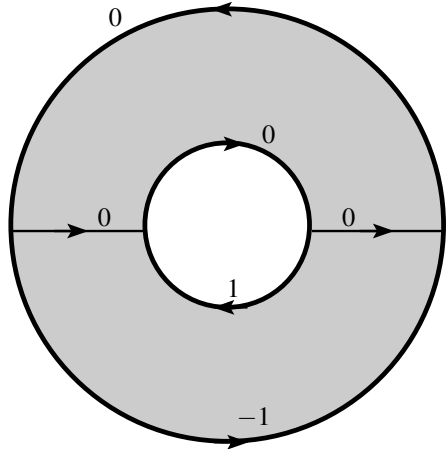}
\caption{An annulus combinatorial surface and a canonical CVF on it that has whirl equal to 0, but is not a tilt field}
\label{Fig: Annulus}
\end{figure}

\begin{challenge}
Prove that, in fact, $\dim H^1(S) = 1$, for the annulus $S$ in Example \ref{annulus}. That is, prove that for any canonical CVF $\vect{G}$ on $S$, there is a VSF $f$ on $(S,G)$ such that 
\[\vect{G} = k\vect{F} + \ob{\tilt} f\]
where $k \in \R$ and $\vect{F}$ is the CVF from the proof of Theorem \ref{annulus}. Indeed, prove that for any combinatorial surface $(S,G)$ with $S \subset \R^2$ having $n$ boundary components, then $\dim H^1(S) = n - 1$. 
\end{challenge}

\subsection{The Good Screen Actor's Guild}
We can perform similar constructions with the original stooges. The resulting cohomology groups are called the \defn{de Rham cohomology groups}.  Let $C^0_{dR}(S)$ denote the vector space of C$^2$ scalar fields on $S$, let $C_{dR}^1(S)$ denote the vector space of C$^1$ vector fields on $S$ and let $C_{dR}^2(S)$ denote the vector space of C$^0$ scalar fields on $S$. We then have the linear maps:
\[
C_{dR}^0(S) \stackrel{\grad}{\to} C_{dR}^1(S) \stackrel{\scurl}{\to} C_{dR}^2(S).
\]
These cochain groups are, unlike the combinatorial versions, infinite dimensional vector spaces. Nonetheless, as before we can form cohomology groups $H^0_{dR}(S)$, $H^1_{dR}(S)$, and $H^2_{dR}(S)$ as before and theorems from algebraic topology tell us that these vector spaces are isomorphic to the vector spaces $H^0(S)$, $H^1(S)$, and $H^2(S)$, respectively. The ambitious reader might like to try to prove this using the techniques from Section \ref{Section: Refining}.

The following traditional result is the analogue of Example \ref{annulus}.
\begin{challenge}
Let $\vect{F}(x,y) = \frac{1}{x^2 + y^2}\begin{pmatrix} -y \\ x \end{pmatrix}$ be a vector field on the annulus $S$ from Theorem \ref{annulus}. Prove that $\scurl \vect{F} = 0$ but that $\vect{F}$ is not in the image of Grad. Furthermore, if $\vect{G}$ is any $C^1$ vector field on $S$ with $\scurl \vect{G} = 0$, prove that there is a constant $k \in \R$ and a C$^2$ scalar field $f$ on $S$ such that $\vect{G} = k\vect{F} + \grad f$. Consequently, $H^1_{dR}(S)$ has dimension 1.
\end{challenge}

\subsection{Further explorations}
In the previous sections, we defined two types of cohomology groups: the combinatorial cohomology groups $H^i(S)$ and the de Rham cohomology groups $H^i_{dR}(S)$. The inquiring reader is bound to ask several questions:
\begin{enumerate}
\item Why are they called cohomology \emph{groups} rather than, say, cohomology \emph{vector spaces}? 
\item Where does the term ``cohomology'' come from and is there such a thing as a homology group?
\item How should all this be generalized to higher dimensions?
\end{enumerate}

The first question is easiest to answer: Just as our creation of the combinatorial cohomology groups doesn't change much when we replace vector spaces over $\R$ with vector spaces over some other field (such as $\C$ or $\Z/2\Z$), so it doesn't change when we replace the vector spaces $C^i(S,G)$ with abelian groups. If you know some abstract algebra, you might enjoy working through the constructions for the case when $\R$ is replaced by $\Z/4\Z$ or some other finite abelian group. The term ``cohomology group'' is derived from this more general setting.

To answer the second and third questions: notice that, in the combinatorial setting, we could also have created vector spaces as follows. Let $C_0(S,G)$ be the vector space with basis in one-to-one correspondence with the vertices of $G$. Let $C_1(S,G)$ be the vector space with basis in one-to-one correspondence with the set of edges of $G$, each with a fixed orientation. Finally, let $C_2(S,G)$ be the vector space with basis in one-to-one correspondence with the faces of $S$, each with fixed orientation. Define linear maps $\boundary_2$ and $\boundary_1$:
\[
C_0(S,G) \stackrel{\boundary_1}{\leftarrow} C_1(S,G) \stackrel{\boundary_2}{\leftarrow} C_2(S,G)
\]
by defining them on the given basis of the domain vector space as follows. For a face $\sigma$ of $G$, let $\boundary_2(\sigma)$ be equal to the sum of the elements of $C_1(S)$ corresponding to the edges of $S$ with the orientation induced by that of $\sigma$. For an oriented edge $e$ of $G$ with initial endpoint $v$ and terminal endpoint $w$, let $\boundary_1(e)$ be equal to the basis element of $C_0(S)$ corresponding to the vertex $w$ minus the basis element of $C_1(S)$ corresponding to $v$. It is easy to verify that $\boundary_1 \circ \boundary_2$ is the zero map and so we define vector spaces $H_0(S) = \im \boundary_1$, $H_1(S) = \ker \boundary_1 / \im \boundary_2$ and $H_2(S) = \ker \boundary_2$. The vector spaces  $H_i(S)$ are called the \defn{homology groups} of $S$ and the vector spaces $C_i(S)$ are called the \defn{chain groups} of $(S,G)$. The cochain groups are the dual vector spaces to the chain groups (hence, the terminology) and the tilt and whirl functions are the dual maps to the ``boundary'' maps $\boundary_1$ and $\boundary_2$. This perspective suggests a way of generalizing the combinatorial setup to higher dimensions: If we have a so-called ``simplicial complex'' (essentially a higher dimensional graph) we can form chain groups and homology groups for the simplicial complex in strict analogy with the surface (2-dimensional) case. The higher dimensional combinatorial cochain groups are then the ``dual'' vector spaces to the chain groups, the higher-dimensional versions of tilt and whirl are the dual maps to the boundary maps and the higher dimensional cohomology groups are quotient vector spaces. In the two-dimensional case it is easy to verify that $H_1(S)$ has the same dimension as $H^1(S)$. The higher dimensional version of this (for objects called ``manifolds'') is known as ``Poincar\'e duality''. Any algebraic topology book (such as \cite{Hatcher}) will have lots more to say about homology and cohomology groups for simplicial complexes and their generalizations.

The best way of generalizing the de Rham cohomology groups to higher dimensions is to develop the theory of ``differential manifolds'' and ``differential forms''. Many vector calculus texts (eg. \cite{Colley, MarsdenTromba}) have a section on differential forms and more can be learned from books such as \cites{Bachman, Edwards}. The corresponding de Rham cohomology theory can be learned from any differential topology book (e.g. \cite{Bredon}).

We conclude this viewer's guide with one last challenge: For most of the article we have ignored divergence and its impersonator ebb,

\begin{challenge}
How do ebb and divergence fit into the theory of combinatorial and de Rham (co)homology groups? Can you prove a ``coarse actor's" version of the Divergence (or Gauss') Theorem from vector calculus?
\end{challenge}

\section{The Credits}
As is abundantly evident to those who know some algebraic or differential topology, apart from the presentation, very little of the mathematics in this article is actually new. We believe, however, that is is useful and pedagogically productive to draw explict analogies, phrased in the language of vector calculus, between the combinatorial and differential versions of cohomology theory. The article was spawned by the difficulty of finding a proof of Green's theorem that illuminated, rather than obscured, the basic ideas, but still applied to a very wide class of surfaces and vector fields. It would be surprising if Theorems \ref{MVTR} and \ref{Thm: COV} were genuinly new, but we have not been able to find them in the literature (although Apostol's proof \cite[Theorem 11.36]{Apostol} is similar to Theorem \ref{Thm: COV}). We also hope that this article will be useful to students encountering homology and cohomology for the first time and that the basic approach to Green's theorem and scalar curl will be helpful to beginning vector calculus students. The third author has successfully used some of these ideas in the classroom. We thank Otto Bretscher and Fernando Gouv\^{e}a for helpful conversations and Colby College for supporting the work of the first two authors.

\end{document}